\documentclass[a4paper,10pt, reqno]{amsart}
\usepackage[a4paper, margin=3 cm]{geometry}
\usepackage{amsfonts}
\usepackage{amsmath}
\usepackage{amssymb}
\usepackage{amsthm}
\usepackage{enumerate}
\usepackage{mathdots}
\usepackage{mathrsfs}
\usepackage{amsbsy}
\usepackage[all]{xy}
\usepackage{bm}
\usepackage{array}
\usepackage{float}
\usepackage{xcolor}
\usepackage{hhline}
\usepackage{mathtools}
\usepackage[colorlinks=true, linkcolor=blue, citecolor=blue, urlcolor=blue]{hyperref}
\usepackage{caption}
\usepackage{subcaption}
\usepackage{adjustbox}
\usepackage{url}
\usepackage[noadjust]{cite}
\usepackage{tabu}
\usepackage{diagbox}
\usepackage{tikz}
\usepackage{bbm}
\usepackage{booktabs}
\usepackage[noabbrev,capitalise]{cleveref}

\newtheorem{theorem}{Theorem}[section]
\newtheorem{lemma}[theorem]{Lemma}
\newtheorem{proposition}[theorem]{Proposition}
\newtheorem{definition}[theorem]{Definition}

\newtheorem*{rmk*}{Remark}

\newtheorem{corollary}[theorem]{Corollary}

\newtheorem{claim}[theorem]{Claim}

\theoremstyle{definition}

\newcommand{\To}{\mathbf{T}}

\allowdisplaybreaks

\newenvironment{enumerate*}%
  {\begin{enumerate}[(I)]%
    \setlength{\itemsep}{10pt}%
    \setlength{\parskip}{0pt}}%
  {\end{enumerate}}

\usepackage{fancyhdr}

\title{Polynomial bounds for the Chowla Cosine Problem}
\author{Benjamin Bedert}
\thanks{bedert.benjamin@gmail.com\\The author gratefully acknowledges financial support from the EPSRC}
\begin{document}
\begin{abstract}
Let $A\subset \mathbf{N}$ be a finite set of $n=|A|$ positive integers, and consider the cosine sum $f_A(x)=\sum_{a\in A}\cos ax$. We prove that $$\min_x f_A(x)\leqslant -n^{1/5-o(1)},$$ thereby establishing polynomial bounds for the Chowla cosine problem.
\end{abstract}
\maketitle

\tableofcontents
\section{Introduction}
Let $A\subset \mathbf{N}$ be a finite set of $n$ positive integers and consider the cosine polynomial $$f_A(x)=\sum_{a\in A}\cos ax.$$ Since $\int_0^{2\pi}f_A(x)\,dx=0$, $f_A$ assumes both strictly positive and strictly negative values. It is clear that $f_A(0)=n$ and $\lVert f_A\rVert_\infty=n$, so it is trivial to determine the largest positive value that $f_A$ assumes. Determining whether $f_A$ must assume large negative values is a hard problem; Ankeny and Chowla \cite{Chowla1952RiemannZeta}, motivated by questions on zeta functions, asked whether for any $K>0$, there is an $n_0$ such that every set $A$ of size $|A|\geqslant n_0$ satisfies $|\min_x f_A(x)|>K$. In 1965, Chowla \cite{chowla1965} posed the more precise question of finding the largest number $K(n)>0$ such that any such cosine polynomial with $n$ terms assumes a value smaller than or equal to $-K(n)$. Chowla's cosine problem is thus to determine $$K(n)=\inf_{A\subset \mathbf{N}:|A|=n}\left( -\min_{x}f_A(x)\right).$$
There exist simple constructions of sets $A$ of size $n$ for which $\sum_{a\in A}\cos ax\geqslant -10\sqrt{n}$ for all $x$. One can for example take $A=\{b_1-b_2:b_j\in B\}\setminus\{0\}$ where $B$ is a Sidon set\footnote{In this context, a set $B$ is said to be \emph{Sidon} if there are no nontrivial solutions to $x_1-x_2=x_3-x_4$ with $x_i\in B$.} of size $m\approx\sqrt{n}$ (and add up to $O(\sqrt{n})$ arbitrary elements if $n$ is not of the form $m^2-m$), and observe that $f_A=\hat{1}_A=|\hat{1}_B|^2-|B|=|\hat{1}_B|^2-O(\sqrt{n})$. This shows that $K(n)\ll \sqrt{n}$, which is the best known upper bound to date; in fact Chowla \cite{chowla1965} conjectured that this is sharp, namely that $K(n)\asymp \sqrt{n}$. 

\medskip

There has been incremental progress on lower bounds for Chowla's cosine problem. The first bound showing that $K(n)\to\infty$ follows from Cohen's work \cite{cohen} on the Littlewood $L^1$ conjecture, as demonstrated by S. and M. Uchiyama \cite{Uchiyama1960CosineProblem}. This was also observed by Roth \cite{Roth1973CosinePolynomials}, who by different methods obtained the stronger bound $K(n)\gg (\log n)^c$ for $c=1/2-o(1)$. We note that the value of this exponent $c$ was later improved as an immediate consequence of various papers on the $L^1$ conjecture, whose resolution by McGehee, Pigno and Smith \cite{mcgeheepignosmith}, and independently Konyagin \cite{konyaginl} ultimately led to $c=1$. Bourgain \cite{Bourgainchowla2,Bourgainchowla1} was the first to breach the logarithmic barrier, establishing the quasipolynomial bound $K(n)>e^{(\log n)^\varepsilon}$ for some $\varepsilon>0$. A further refinement of Bourgain's method by Ruzsa \cite{ruzsachowla} shows that $K(n)\geqslant e^{c'(\log n)^{1/2}}$, which stood as the previous record. We also mention that, among other results, Sanders \cite{sanders2010chowla} proved a polynomial bound for $|\min_x f_A(x)|$ in the special setting where all elements of $A$ have size $O(n)$. Our main result is the following improvement, producing the first polynomial bounds for the Chowla cosine problem.
\begin{theorem}\label{th:chowlapoly}
Any set $A\subset \mathbf{N}$ of $n=|A|$ positive integers satisfies $$\min_{x\in[0,2\pi]} \sum_{a\in A}\cos ax\leqslant -n^{1/5-o(1)}.$$ 
\end{theorem}
Hence, $K(n)\gg n^{1/5-o(1)}$. For clarity of exposition, we first prove this result with the weaker (but still polynomial) bound $K(n)\gg n^{1/12}$ using a streamlined version of our argument in Section~\ref{sec:chowla1/12}, which takes up only 5 pages. In the final Section~\ref{sec:exponentvalue}, we show that this may be improved to $K(n)\gg n^{1/5-o(1)}$ with some further effort. It seems plausible that our method can produce an even better exponent than $1/5$, but it remains an interesting open problem to investigate whether the bound $K(n)\asymp n^{1/2}$ is true. 
\par We remark that very recently, Jin, Milojevi{\'c}, Tomon and Zhang \cite{jin2025smallEigenvalues} independently uploaded a preprint that also establishes a polynomial bound $K(n)\gg n^c$
with exponent $c= 1/10-o(1)$. The method in our paper differs significantly from that of Jin et al., whose main theorem is a structural result about graphs with no small (i.e.~large and negative) eigenvalues.

\medskip

We also mention that all other existing methods which yield superlogarithmic bounds for $K(n)$, including those in the work \cite{jin2025smallEigenvalues} of Jin et al., are very sensitive to the cosine polynomials having all their coefficients in $\{0,1\}$. Let $S$ be a subset of $\mathbf{R}\setminus\{0\}$ and write $\mathcal{C}_S(n)$ for the class of cosine polynomials with $n$ terms and coefficients in $S$. Then one can define the analogous quantity
$$K_S(n)=\inf_{f\in\mathcal{C}_S(n)}\left(-\min_{x\in[0,2\pi]}f(x)\right).$$ 
The previous strongest general bound in the literature states that $K_S(n)\gg (\min_{s\in S}|s|)\log n$, which follows from a simple application of the $L^1$ conjecture (as in McGehee-Pigno-Smith \cite{mcgeheepignosmith}). Our methods provide polynomial bounds in the general setting where $S$ is an arbitrary finite set. 
\begin{theorem}\label{th:generalcoefintro}
Let $S\subset\mathbf{R}\setminus\{0\}$ be finite. Then there exist two constants $c_S,c'_S>0$ such that the following holds. For every symmetric set $A\subset\mathbf{Z}\setminus\{0\}$ of size $n=|A|$, and every choice of coefficients $s_a\in S$ satisfying $s_a=s_{-a}$ for all $a\in A$, we have that \begin{equation*}
        \min_{x\in\mathbf{R}}\sum_{a\in A}s_ae(ax)\leqslant -c'_Sn^{c_S}.
    \end{equation*}
\end{theorem}
This result only applies when the set of coefficients $S$ has fixed size (or sufficiently small size compared to $n$). In analogy with the $L^1$ conjecture \cite{mcgeheepignosmith}, one might wonder whether $K_S(n)$ exhibits polynomial growth whenever $\min_{s\in S}|s|\gg 1$, irrespective of the size of $S$. This is false however:~a rather deep construction of Belov and Konyagin \cite{konyaginbelov} shows that there exist cosine polynomials $f(x)=\sum_{j=1}^na_j\cos(jx)$ with positive integer coefficients in $S=\{1,2,\dots,O(\log n)^3\}$ for which $\min_x f(x)\geqslant -O(\log n)^3$. One can alternatively interpret this as saying that the conclusion of Theorem~\ref{th:chowlapoly} fails dramatically if one considers \emph{multisets} $A$ of integers of size $n$ (whereas Chowla's cosine problem only deals with genuine sets of $n$ distinct integers). Finally, we reiterate an interesting open question of Ruzsa, which is to estimate $K_{[0.99,1.01]}(n)$, namely negative values of cosine polynomials with coefficients in the interval $[0.99,1.01]$. Seemingly the best known result for Ruzsa's problem is still the rather weak bound $K_{[0.99,1.01]}(n)\gg \log n$ which follows from a simple application of the resolution of the $L^1$ conjecture in \cite{mcgeheepignosmith}. 
\par \textbf{Acknowledgements:} The author would like to thank Thomas F. Bloom for sharing several insightful perspectives around Chowla's cosine problem and related questions.

 \section{Notation and prerequisites}
We use the asymptotic notation $f=O(g)$, $f\ll g$, or $g=\Omega(f)$ if there is an absolute constant $C$ such that $|f(y)|\leqslant C g(y)$ for all $y$ in a certain domain which will be clear from context. We write $f=o(g)$ if $f(y)/g(y)\to 0$ as $y\to\infty$, and we write $f\asymp g$ if both $f\ll g$ and $g\ll f$. We sometimes include an extra subscript such as $f=O_S(g)$ to indicate that the implied constant $C$ is allowed to depend on $S$.

\smallskip

We write $\mathbf{N},\mathbf{Z},\mathbf{R}$ and $\mathbf{C}$ for the natural, integer, real, and complex numbers, respectively. For two sets $E_1,E_2$, we define the sumset $E_1+E_2:=\{e_1+e_2:e_1\in E_1, e_2\in E_2\}$ and difference set $E_1-E_2=\{e_1-e_2:e_1\in E_1,e_2\in E_2\}$. We use the standard notation $e(x)=e^{2\pi i x}$ for $x\in \mathbf{R}$, and, as this function is $1$-periodic, it is natural to consider the domain of the variable $x$ to be $\mathbf{T}=\mathbf{R}/\mathbf{Z}$. For a suitably integrable function $g:\mathbf{T}\to\mathbf{C}$ we denote, for $p\in[1,\infty)$, its $L^p$-norm by $$\lVert g\rVert_p\vcentcolon= \left(\int_0^1|g(x)|^p\,dx\right)^{1/p},$$ and $\lVert g\rVert_\infty$ is the smallest constant $M$ such that $|g(x)|\leqslant M$ holds almost everywhere. Its Fourier transform is the function $\hat{g}:\mathbf{Z}\to\mathbf{C}$ which is defined by $\hat{g}(n)=\int_\mathbf{T}g(x)e(-nx)\,dx$. If $f:\mathbf{Z}\to\mathbf{C}$ is a function, we shall denote its Fourier transform by $\hat{f}(x)=\sum_{n\in\mathbf{Z}}f(n)e(nx)$ which is a priori simply a formal series. In this paper, $\hat{f}:\mathbf{T}\to\mathbf{C}$ will always be a trigonometric polynomial. Of specific importance are the Fourier transforms of (indicator functions of) finite sets $A\subset\mathbf{Z}$, and we write $\hat{1}_A(x):=\sum_{a\in A}e(ax)$.

\medskip

For two functions $g,h\in L^2(\mathbf{T})$ we define $\langle g,h\rangle =\int_\mathbf{T} g(x)\overline{h(x)}\,dx$ and we shall frequently make use of Parseval's theorem which states that $\langle g,h\rangle =\sum_{n\in\mathbf{Z}}\hat{g}(n)\overline{\hat{h}(n)}$. We also define their convolution to be the function $g*h:\To\to\mathbf{C}$ given by $(g*h)(x) = \int_\mathbf{T}g(x-y)h(y)\,dy$, and we note the basic fact that $\widehat{g*h}(n) = \hat{g}(n)\hat{h}(n)$. This formula for the Fourier coefficients of a convolution plays an important role throughout this paper, as it implies crucial properties such as that $(\hat{1}_A*\hat{1}_B)(x)=\hat{1}_{A\cap B}(x)$ for any two finite $A,B\subset \mathbf{Z}$. Finally, Young's convolution inequality states that 
\begin{equation}\label{eq:Young}
\lVert g*h\rVert_r\leqslant\lVert g\rVert_p\lVert h\rVert_q    
\end{equation} whenever $p,q,r\in[1,\infty]$ satisfy $1+1/r=1/p+1/q$.

\section{Preliminary observations}
Note that for a set $B\subset \mathbf{N}$ of positive integers, we may write $$2\sum_{b\in B}\cos(2\pi bx)=\sum_{a\in B\cup-B}e(ax)=\hat{1}_A(x),$$ where $A=B\cup -B$. We may therefore consider the following equivalent setup of the Chowla cosine problem, which is notationally more convenient. Let $A\subset\mathbf{Z}\setminus\{0\}$ be symmetric, meaning that $A=-A$, so that $\hat{1}_A(x)$ is a real-valued function on $\mathbf{T}=\mathbf{R}/\mathbf{Z}$ (and a sum of $|A|/2$ cosines). Let $K>0$ be a constant such that $$\hat{1}_A(x)+K=\sum_{a\in A}e(ax)+K\geqslant 0,\quad \forall x\in\mathbf{T}.$$ Our aim is to show that $K$ is large in terms of $n=|A|$, and in order to establish Theorem~\ref{th:chowlapoly}, we need to show that $K\gg n^{1/5-o(1)}$. We introduce some convenient notation.
\begin{definition}\normalfont
    For a real-valued function $g\in L^1(\mathbf{T})$ with $\int_\mathbf{T}g(x)\,dx=0$, we define
    \begin{equation}
        \lVert g \rVert_{\min}=-\operatorname{ess\, inf}_{x\in \mathbf{T}}g(x).
    \end{equation} Equivalently, whenever this essential infimum is finite, $\lVert g \rVert_{\min}$ is the smallest number $K$ such that $g(x)+K$ is nonnegative for almost all $x\in \To$. 
\end{definition}
The assumption that $\int_\mathbf{T}g(x)\,dx=0$ guarantees that $\lVert g\rVert_{\min}$ is nonnegative. We may decompose any real-valued function $g(x)$ as $g=g^+-g^-$ where $g^+(x)=\max(g(x),0)$ and $g^-(x)=\max(-g(x),0)$ are nonnegative functions. It is clear that $\lVert g\rVert_{\min}=\lVert g^-\rVert_\infty$ is a one-sided estimate for $g$, so that $\lVert \cdot \rVert_{\min}$ satisfies the triangle inequality and that $\lVert \lambda g\rVert_{\min}=\lambda\lVert g\rVert_{\min}$ for $\lambda>0$ (although it is not a norm since this is not well-behaved under dilations by negative $\lambda$). We require some useful basic lemmas.
\begin{lemma}\label{lem:mintoL^1}
    Let $g\in L^1(\mathbf{T})$ be a real-valued function with $\int_\To g(x)\,dx = 0$. Then
    $\lVert g \rVert_{1}\leqslant 2\lVert g \rVert_{\min}.$
\end{lemma}
\begin{proof}
    By definition, $g(x)\geqslant -\lVert g \rVert_{\min}$ holds for almost all $x$. Hence, the bound $|g(x)|\leqslant 2\lVert g \rVert_{\min} +g(x)$ also holds a.e. Integrating this inequality over $\mathbf{T}$ gives the result because $\int_\To g(x)\,dx=0$.
\end{proof}
The following two results show how convolution interacts with $\lVert\cdot\rVert_{\min}$.
\begin{lemma}\label{lem:crucial}
    Let $g,h\in L^2(\mathbf{T})$ be real-valued functions with $\int_\To g = \int_\To h = 0$. Then
    \begin{equation*}\label{eq:convminL1}
        \lVert g*h \rVert_{\min}\leqslant \frac{\lVert g \rVert_{\min}\lVert h \rVert_1+\lVert g \rVert_{1}\lVert h \rVert_{\min}}{2}.
    \end{equation*}
\end{lemma}
\begin{proof}
    As $g,h\in L^2(\mathbf{T})$, their convolution $g*h$ is well-defined pointwise, and even continuous on $\mathbf{T}$. We write $g^+=\max(g,0)$ and $g^-=\max(-g,0)$ so that $g=g^+-g^-$ and $|g|=g^++g^-$, and similarly for $h$. As $g^+,g^-,h^+,h^-$ are pointwise nonnegative, we may estimate
    \begin{align*}
        (g*h)(x)&=\int_\To (g^+(u)-g^-(u))(h^+(x-u)-h^-(x-u))\, du\\
        &\geqslant -\int_\To g^-(u)h^+(x-u)\, du - \int_\To g^+(u)h^-(x-u)\, du\\
        &\geqslant -\lVert g \rVert_{\min}\int_\To h^+(u)\,du -\lVert h \rVert_{\min}\int_\To g^+(u)\,du .
    \end{align*}
    Note that 
    $\int_\To g^+(u)\,du=\int_\To g^-(u)\,du$ because of the assumption that $\int_\To g(u)\,du = 0$. As $$\lVert g\rVert_1=\int_\To g^+(u)\,du+\int_\To g^-(u)\,du,$$ we therefore deduce that $\int_\To g^+(u)\,du=\lVert g\rVert_1/2$, and the analogous result for $h$. Plugging this into the inequality above shows that $(g*h)(x)\geqslant -\frac{1}{2}\left(\lVert g \rVert_{\min}\lVert h \rVert_1+\lVert g \rVert_{1}\lVert h \rVert_{\min}\right)$ for all $x\in\To$.
\end{proof}
The next lemma provides a simpler one-sided estimate for the convolution of two functions in terms of their individual one-sided estimates. We note that, up to an extra factor of 2, this may be deduced by combining the previous two lemmas.
\begin{lemma}\label{lem:Kconv}
    Let $g,h\in L^2(\mathbf{T})$ be real-valued functions with $\int_\To g = \int_\To h = 0$. Then
    $$\lVert g*h \rVert_{\min}\leqslant \lVert g \rVert_{\min}\lVert h \rVert_{\min}.$$
\end{lemma}
\begin{proof}
    Note that the functions $g(x)+\lVert g \rVert_{\min}$ and $h(x)+\lVert h \rVert_{\min}$ are nonnegative a.e. Hence, so is their convolution which, as $\int_\To g(x)\,dx = \int_\To h(x)\,dx = 0$, has the form
    \begin{align*}
        \left(g+\lVert g \rVert_{\min}\right)*\left(h+\lVert h \rVert_{\min}\right) &= g*h + \lVert g\rVert_{\min}\lVert h \rVert_{\min}.
    \end{align*}
    We deduce that $g*h(x)\geqslant -\lVert g\rVert_{\min}\lVert h \rVert_{\min}$ for all $x\in \To$.
    
\end{proof}

\section{Arithmetic results from Roth, Bourgain and Ruzsa}
We will make use of two intermediate results from the approaches of Roth \cite{Roth1973CosinePolynomials}, Bourgain \cite{Bourgainchowla1}, and Ruzsa \cite{ruzsachowla}. Both results indicate that sets $A$ either contain or lack certain types of arithmetic structure under the assumption that $\lVert\hat{1}_A\rVert_{\min}$ is small, this being a major theme in all three papers.

\medskip

The first is due to Roth, who showed (something slightly stronger than) that $A$ has large additive energy under the assumption that $\lVert \hat{1}_A\rVert_{\min}$ is not too large. For the convenience of the reader, we reproduce the short proof here.
\begin{lemma}[\cite{Roth1973CosinePolynomials}, Lemma 5]\label{lem:roth}
    Let $A=-A\subset\mathbf{Z}\setminus\{0\}$ be finite, and suppose that $\hat{1}_A(x)+K\geqslant 0$. Then every subset $B\subseteq A$ of size $|B|\geqslant 2K^2$ satisfies $$\#\{(b_1,b_2)\in B^2:b_1-b_2\in A\}\geqslant \frac{|B|^2}{2K}.$$ 
\end{lemma}
\begin{proof}
    Note that the assumption implies that $|\hat{1}_A(x)+K|=\hat{1}_A(x)+K$. By applying the Cauchy-Schwarz inequality to the functions $\hat{1}_B(x)|\hat{1}_A(x)+K|^{1/2}$ and $|\hat{1}_A(x)+K|^{1/2}$, we see that
    \begin{align*}
        |B|&=\langle \hat{1}_B,\hat{1}_A+K\rangle\\
        &\leqslant \left(\int_\To (\hat{1}_A(x)+K)\,dx\right)^{1/2}\left( \int_\To |\hat{1}_B(x)|^2(\hat{1}_A(x)+K)\,dx\right)^{1/2}\\
        &=K^{1/2}\bigg( \#\{(b_1,b_2)\in B^2:b_1-b_2\in A\}+K|B|\bigg)^{1/2},
    \end{align*}
    where the two equalities are applications of Parseval.
    Rearranging gives the inequality $$\#\{(b_1,b_2)\in B^2:b_1-b_2\in A\}\geqslant \frac{|B|^2}{K}-K|B|,$$ which yields the desired result since $|B|\geqslant 2K^2$.
\end{proof}
We have taken the following lemma from Ruzsa's paper, though very similar results already appear in those of Roth and Bourgain. We shall prove a more general version of this lemma in Lemma \ref{lem:ruzsachowlaS}. 
\begin{lemma}[\cite{ruzsachowla}, Lemma 3.1]\label{lem:ruzsadifference}
    Let $A=-A\subset\mathbf{Z}\setminus\{0\}$ be finite. Let $U,V\subset \mathbf{Z}$ be such that $U-V+\{0,d\}\subset A$ for some $d\neq 0$. Then $$\lVert \hat{1}_A\rVert_{\min}\geqslant\frac12 \sqrt{\min(|U|,|V|)}.$$
\end{lemma}
We will require the following immediate corollary for Theorem~\ref{th:chowlapoly}. 
\begin{corollary}\label{cor:ruzsachowla}
Let $A=-A\subset\mathbf{Z}\setminus\{0\}$ be finite. Suppose that the arithmetic progression $P$ is contained in $A$. Then $|P|\ll \lVert \hat{1}_A\rVert_{\min}^2$.
\end{corollary} 
To deal with the general setting in Theorem~\ref{th:generalcoefintro}, where we prove one-sided estimates for cosine polynomials with coefficients in an arbitrary finite set $S\subset\mathbf{R}\setminus\{0\}$, we need the following extension of Corollary~\ref{cor:ruzsachowla}. The proof is a rather natural adaptation of the argument of Ruzsa.
\begin{lemma}\label{lem:ruzsachowlaS}
Let $S=\{s_1>s_2>\dots>s_k\}\subset\mathbf{R}\setminus\{0\}$ be finite, and $s_1>0$. Suppose that $A^{(1)},\dots,A^{(k)}\subset\mathbf{Z}\setminus\{0\}$ are pairwise disjoint, finite, and symmetric. If $P$ is an arithmetic progression that is contained in $A^{(1)}$, then $$|P|\ll_S\left\lVert \sum_{j=1}^ks_j\hat{1}_{A^{(j)}}\right\rVert_{\min}^2.$$    
\end{lemma}
\begin{proof}
    Let $P=\{x_0-pd,x_0-(p-1)d,\dots, x_0+pd\}$ be an arithmetic progression with common difference $d$ and size $|P|\asymp p$, and assume that $P\subset A^{(1)}$. We define the subprogression $U_0:=\{x_0+d,\dots,x_0+pd\}$, and observe that $U_0-(-x_0+U_0)+\{0,d\}\subset P\subset A^{(1)}$. Now consider the translates $U_\ell:= U_0+\ell d$ of $U_0$ by multiples of $d$, for $\ell\in\mathbf{Z}$. Since $A^{(1)}$ is a finite set, it is clear that $U_\ell \cap A^{(1)}=\emptyset$ for all sufficiently large $\ell$. On the other hand, $U_0\cap A^{(1)}=U_0$ has size $p$ by the definition of $U_0$, and hence there must exist some nonnegative integer $\ell$ for which
    \begin{align*}
    |U_\ell\cap A^{(1)}|&\geqslant p/2,\\
    |U_{\ell+1}\cap A^{(1)}|&<p/2.
    \end{align*}
    If $|(-x_0+U_{\ell})\cap A^{(1)}|< p/2$ then we define $U:= U_\ell\cap A^{(1)}$ and $V:=(-x_0+U_{\ell})\setminus A^{(1)}$. In the remaining case where $|(-x_0+U_{\ell})\cap A^{(1)}|\geqslant p/2$, we define $U:= (-x_0+U_{\ell})\cap A^{(1)}$ and $V:=U_{\ell+1}\setminus A^{(1)}$. In both of these scenarios, we have therefore found sets $U,V$ which each have size at least $p/2$, and where $U\subset A^{(1)}$ and $V\cap A^{(1)}=\emptyset$. As we chose $U_0$ so that $U_0-(-x_0+U_0)+\{0,d\}\subset A^{(1)}$, we also have the crucial property that $U-V\subset A^{(1)}$. We may remove some elements from $U,V$ so that, additionally, $0\notin U,V$ and $|U|=|V|\geqslant p/2-1$. Let us write $F=\sum_{j=1}^ks_j\hat{1}_{A^{(j)}}$, and $K=\lVert F\rVert_{\min}$. Then by Parseval, we have
    \begin{align*}
        \int_\To(\hat{1}_{U}(x)-\hat{1}_V(x))(F(x)+K)\,dx&=\sum_{u\in U}\hat{F}(u)-\sum_{v\in V}\hat{F}(v)\\
        &\geqslant s_1|U|-\max(s_2,0)|V|\gg (s_1-\max(s_2,0))p\gg_S p.
    \end{align*}
    The assumption that $F(x)+K$ is pointwise nonnegative allows us to use the Cauchy-Schwarz inequality, giving the upper bound
    \begin{align*}
        &\int_\To(\hat{1}_{U}(x)-\hat{1}_V(x))(F(x)+K)\,dx\\
        &\leqslant \left(\int_\To(F(x)+K)\,dx\right)^{1/2}\left(\int_\To(F(x)+K)(|\hat{1}_U|^2+|\hat{1}_V|^2-\hat{1}_U\hat{1}_{-V}-\hat{1}_V\hat{1}_{-U})\,dx\right)^{1/2}\\
        &=K^{1/2}\left( K(|U|+|V|)+\sum_{u_i\in U:u_1\neq u_2}\hat{F}(u_1-u_2)-2\sum_{u\in U,v\in V}\hat{F}(u-v)+\sum_{v_i\in V:v_1\neq v_2}\hat{F}(v_1-v_2)\right)^{1/2}\\
        &\leqslant K^{1/2}\left( K|U|+K|V|\right)^{1/2}\ll K p^{1/2},
    \end{align*}
    where we used Parseval, together with the facts that $U-V\subset A^{(1)}$ and that all Fourier coefficients of $F$ are at most $s_1$. Combining these inequalities shows that $p\ll_S K^2$, as we claimed.
\end{proof}
\section{Polynomial bounds for the Chowla cosine problem}\label{sec:chowla1/12}
In this section, we provide a streamlined argument which establishes polynomial bounds for Chowla's cosine problem, proving Theorem~\ref{th:chowlapoly} with the slightly smaller exponent $1/12$. We suppose throughout this section that $A=-A\subset\mathbf{Z}\setminus\{0\}$ is a finite symmetric set of integers of size $n=|A|$, and that $K>0$ is a constant such that $\hat{1}_A(x)+K\geqslant 0$ for all $x\in \To$.
By Roth's Lemma~\ref{lem:roth}, we may suppose that $$\sum_{t\in A}|A\cap(A+t)|=\#\{(a,t)\in A^2:a-t\in A\}\geqslant \frac{n^2}{2K}$$ since otherwise
$n\leqslant (2K)^{2}$, which is stronger than the desired
conclusion. Hence, we can find a $t\neq 0$ such that $A_t:=A\cap (A+t)$ has size 
\begin{equation}\label{eq:A_tsize}
    |A_t|\geqslant n/(2K).
\end{equation} We remark that the existence of these large sets $A_t$ also plays a central role in the arguments of Roth, Bourgain and Ruzsa, though for a different reason.

\medskip

Consider a fixed $t\in \mathbf{Z}\setminus\{0\}$. The function $1+\sin(2\pi tx)$ is clearly nonnegative for all $x\in \To$. Together with the assumption that $\hat{1}_A(x)+K\geqslant 0$ for all $x\in \To$, this implies that $$(1+\sin(2\pi tx))(\hat{1}_A(x)+K)\geqslant 0.$$ As $|1+\sin(2\pi tx)|\leqslant 2$, this in turn shows that $$(1+\sin(2\pi tx))\hat{1}_A(x)+2K\geqslant 0, \quad \forall x\in \To.$$ We note that $\sin(2\pi tx)\hat{1}_A(x)=\frac{1}{2i}\hat{1}_{A+t}+\frac{-1}{2i}\hat{1}_{A-t}$ and hence we may rewrite this as
$$\hat{1}_A(x)+\frac{1}{2i}\cdot\hat{1}_{A+t}(x)+\frac{-1}{2i}\cdot\hat{1}_{A-t}(x)+2K\geqslant 0.$$ It is convenient to rescale the function in the previous inequality by a factor of 2 and define $f_t(x)\vcentcolon=2(1+\sin(2\pi tx))\widehat{1_A}(x)$, so we observe from the expression above that the Fourier coefficients $\widehat{f_t}(m)$ are given by \begin{align}\label{eq:Ffouriercoef}
    \widehat{f_t}(m)=\begin{cases}
        2&\text{if }m\in A\text{, and either }m\notin (A+t)\cup(A-t)\text{ or }m\in(A+t)\cap(A-t),\\
        2\pm \frac1{i}&\text{if }m\in A\text{, and } m\in(A\pm t)\setminus (A\mp t),\\
        \frac{\pm1}{i}&\text{if }m\notin A\text{, and }m\in(A\pm t)\setminus (A\mp t),\\
        0&\text{otherwise.}
    \end{cases}
\end{align}
Note also that we have shown that $\lVert f_t\rVert_{\min}\leqslant 4K$. Define $\lambda:=2-i$. The fact that $A$ is symmetric implies that $A_{-t}:=A\cap(A-t)=-A\cap -(A+t)=-A_t$, so we can rewrite \begin{align}\label{eq:fminbound}
    f_t(x)&=\lambda\cdot\hat{1}_{A_t\setminus -A_t}(x)+\overline{\lambda}\cdot\hat{1}_{-A_t\setminus A_t}(x)+2\cdot\hat{1}_{A\setminus (A_t\triangle -A_{t})}(x)\\ &-i\cdot\hat{1}_{(A+t)\setminus (A\cup(A-t))}(x)+i\cdot\hat{1}_{(A-t)\setminus (A\cup (A+t))}(x),\nonumber
\end{align}
where $A_t\triangle -A_{t}=(A_t\setminus -A_t)\cup(-A_t\setminus A_t)$ denotes the symmetric difference. What we have achieved by writing $f_t$ in this form is that the terms have disjoint Fourier spectra, meaning that the five sets $$A_t\setminus -A_t, -A_t\setminus A_t, A\setminus (A_t\triangle -A_{t}), (A+t)\setminus (A\cup(A-t)), (A-t)\setminus (A\cup (A+t))$$ are pairwise disjoint. We note that $\lambda=2-i$ is a constant (independent of $A$ and $t$) and that its precise value is unimportant for the purpose of getting \emph{some} polynomial bound for Chowla's cosine problem; we only require that $|\lambda|>2$ and that it has nonzero imaginary part. By using the inequality $(1-\sin(2\pi t x))\hat{1}_A(x)+2K\geqslant 0$ instead of $(1+\sin(2\pi t x))\hat{1}_A(x)+2K\geqslant 0$, we obtain the analogous result that $\lVert g_t\rVert_{\min}\leqslant 4K$ for the function 
\begin{align*}
    g_t(x)\vcentcolon=&\overline{\lambda}\cdot\hat{1}_{A_t\setminus -A_t}(x)+\lambda\cdot\hat{1}_{-A_t\setminus A_t}(x)+2\cdot\hat{1}_{A\setminus (A_t\triangle -A_{t})}(x)\\ &+i\cdot\hat{1}_{(A+t)\setminus (A\cup(A-t))}(x)-i\cdot\hat{1}_{(A-t)\setminus (A\cup (A+t))}(x).\nonumber
\end{align*}
Note that the Fourier coefficients of $f_t$ are the conjugates of those of $g_t$, i.e.~$\widehat{f_t}(m)=\overline{\widehat{g_t}(m)}$.
\begin{lemma}\label{lem:basictextra}
    Let $A\subset \mathbf{Z}\setminus\{0\}$ be a finite symmetric set. Suppose that there exists a constant $K>0$ such that $\hat{1}_A(x)+K\geqslant 0$ for all $x\in \To$. Let $t\in \mathbf{Z}\setminus\{0\}$. Then the functions $f_t,g_t$ satisfy $\lVert f_t\rVert_{\min},\lVert g_t\rVert_{\min}\leqslant 4K$. 
\end{lemma}
This leads us to the following proposition. In the proof, we will write $h^{(*m)}=\underbrace{h*\dots*h}_m$ for the $m$-fold convolution of a function $h$. It is also convenient to introduce some notation and write 
\begin{align}\label{eq:A_tB_tC_tdefi}
    A_t&:=A\cap(A+t),\\
    B_t&:=A_t\setminus-A_t=(A\cap (A+t))\setminus(A-t),\nonumber\\
    C_t&:= (A+t)\setminus (A\cup (A-t)),\nonumber\\
    D_t&:=A\setminus(A_t\triangle-A_t) .\nonumber
\end{align} Recall that $A_{-t}=A\cap (A-t)=-A_t$ as $A$ is symmetric, and similarly $B_{-t}=-B_t$ and $C_{-t}=-C_t$. The set $D_t$ is itself symmetric. This notation allows us to simplify the expression \eqref{eq:fminbound} for $f_t$ as follows:
\begin{align}\label{eq:f_tsimp}
     f_t(x)=&\lambda\cdot\hat{1}_{B_t}(x)+\overline{\lambda}\cdot\hat{1}_{-B_t}(x)+2\cdot\hat{1}_{D_t}(x)\\ &-i\cdot\hat{1}_{C_t}(x)+i\cdot\hat{1}_{-C_t}(x),\nonumber
\end{align}
and analogously for $g_t$. In the statement of the next proposition, it is helpful to recall that the sets $B_t,-B_t,D_t,C_t$ and $-C_t$ are pairwise disjoint, which can easily be checked from their definitions.
\begin{proposition}\label{prop:absoluteboundedextra}
    Let $A\subset \mathbf{Z}\setminus\{0\}$ be a finite symmetric set satisfying the inequality $\hat{1}_A(x)+K\geqslant 0$ for all $x\in \To$. Let $t\in\mathbf{Z}\setminus\{0\}$.     Then \begin{align*}
    &|11(\hat{1}_{B_t}-\hat{1}_{-B_t})(x)-(\hat{1}_{C_t}-\hat{1}_{-C_t})(x)| \\
    &\leqslant 8\cdot\hat{1}_{D_t}(x)+ 2(\hat{1}_{B_t}+\hat{1}_{-B_t})(x)+O(K^{3}),\quad \forall x\in \To.
    \end{align*}
\end{proposition}
\begin{rmk*}
    As will become clear later, the exact shape of various terms in this inequality is not too important. The one crucial feature that we desire is that the Fourier coefficient of the term $\hat{1}_{B_t}$ on the left-hand side (which is 11) is strictly larger than all Fourier coefficients of the function on the right-hand side (except the constant term).
\end{rmk*}
\begin{proof}
It is easy to see that $\int_\To f_t(x)\,dx=\int_\To g_t(x)\,dx=0$, so by Lemmas \ref{lem:basictextra} and \ref{lem:Kconv} we obtain
\begin{align}\label{eq:f_tg_tmin}
    \lVert f_t^{(*3)}\rVert_{\min}\ll K^3,\\
    \lVert g_t^{(*3)}\rVert_{\min}\ll K^3.\nonumber
\end{align} From \eqref{eq:f_tsimp}, and as convolving two functions corresponds to multiplying their Fourier coefficients, we see that these convolutions have the following explicit expressions: 
\begin{align*}
     f_t^{(*3)}(x)=&\lambda^3\cdot\hat{1}_{B_t}(x)+\overline{\lambda}^3\cdot\hat{1}_{-B_t}(x)+8\cdot\hat{1}_{D_t}(x)\\ &+i\cdot\hat{1}_{C_t}(x)-i\cdot\hat{1}_{-C_t}(x),\nonumber\\
     g_t^{(*3)}(x)=&\overline{\lambda}^3\cdot\hat{1}_{B_t}(x)+\lambda^3\cdot\hat{1}_{-B_t}(x)+8\cdot\hat{1}_{D_t}(x)\\ &-i\cdot\hat{1}_{C_t}(x)+i\cdot\hat{1}_{-C_t}(x).\nonumber
\end{align*} Note that $\lambda^3=(2-i)^3=2-11i$ and $\overline{\lambda}^3=2+11i$. Let us use this to rewrite our expressions for $f_t^{(*3)}$ and $g_t^{(*3)}$, so that upon recalling \eqref{eq:f_tg_tmin}, we obtain the new inequalities
\begin{align*}
     &-11i(\hat{1}_{B_t}-\hat{1}_{-B_t})(x)+2(\hat{1}_{B_t}+\hat{1}_{-B_t})(x)+8\cdot\hat{1}_{D_t}(x)\\ &+i(\hat{1}_{C_t}-\hat{1}_{-C_t})(x)\geqslant -O(K^3), \quad \forall x\in \To
\end{align*}
and
\begin{align*}
     &11i(\hat{1}_{B_t}-\hat{1}_{-B_t})(x)+2(\hat{1}_{B_t}+\hat{1}_{-B_t})(x)+8\cdot\hat{1}_{D_t}(x)\\ &-i(\hat{1}_{C_t}-\hat{1}_{-C_t})(x)\geqslant -O(K^3), \quad \forall x\in \To.
\end{align*}
Note that these two inequalities have precisely the same terms on their left-hand sides, except that the sign of the term $11i(\hat{1}_{B_t}-\hat{1}_{-B_t})-i(\hat{1}_{C_t}-\hat{1}_{-C_t})$ is flipped. Observe also that this term is a real-valued function on $\To$, as it equals $-2\cdot\Im\left(11\cdot\hat{1}_{B_t}-\hat{1}_{C_t}\right) $. Hence, they combine to show the desired inequality that 
\begin{align*}
    &|11(\hat{1}_{B_t}-\hat{1}_{-B_t})(x)-(\hat{1}_{C_t}-\hat{1}_{-C_t})(x)|
    \leqslant 8\cdot\hat{1}_{D_t}(x)+ 2(\hat{1}_{B_t}+\hat{1}_{-B_t})(x)+O(K^{3}),\quad \forall x\in \To.
    \end{align*}
\end{proof}

The inequality in Proposition~\ref{prop:absoluteboundedextra} provides rather strong information, and we will see that if $B_t=A_t\setminus-A_t$ is sufficiently large, then it can be used to obtain a polynomial bound for $K$. Recall from \eqref{eq:A_tsize} that the assumption that $\hat{1}_A(x)+K\geqslant0$ allows us to find a $t\neq0$ so that $A_t=A\cap(A+t)$ has size $|A_t|\geqslant n/(2K)$. For Proposition~\ref{prop:absoluteboundedextra}, we need to upgrade this to the stronger conclusion that the set $B_t=A_t\setminus-A_t=(A\cap (A+t))\setminus(A-t)$ is large.
\begin{lemma}\label{prop:B_tlarge}
Let $A'\subset\mathbf{Z}\setminus\{0\}$ be a symmetric finite set, and suppose that every arithmetic progression $P$ which is contained in $A'$ has size at most $L$. Let $A'_t:=A'\cap (A'+t)$, where $t$ is a nonzero integer. Then $B'_t:=A'_t\setminus-A'_t$ has size $|B'_t|\geqslant |A'_t|/L$.         
\end{lemma}
\begin{proof}
    As $A'$ is symmetric, we may without loss of generality assume that $t>0$. First, we note that $-A'_t=-A'\cap-(A'+t)=A'\cap(A'-t)=A'_{-t}$ as $A'$ is symmetric. Now partition $A'=\bigsqcup_{i}P_i$ into the minimum possible number of arithmetic progressions with common difference $t$,  meaning that the $P_i\subset A'$ are arithmetic progressions with common difference $t$ for which $(\max P_i) +t\notin A'$ and $(\min P_i)-t\notin A'$. We will show that every progression $P_i$ which has size at least two contributes a unique element to $B'_t=A'_t\setminus -A'_t$. Indeed it is clear that $(\max P_i)\in A'_t$ as $(\max P_i),(\max P_i)-t\in A'$ when $P_i$ has size at least two, whereas the fact that $(\max P_i)+t\notin A'$ by our definition of the $P_i$ shows that $(\max P_i)\notin A'_{-t}=-A'_t$. Let $Q_1,\dots, Q_{r}$ be the collection of progressions $P_i$ which have size at least two, so we have just shown that 
\begin{equation}\label{eq:B_tL}
        |B'_t|\geqslant r.
    \end{equation} 
Note on the other hand that if $a\in A'_t=A'\cap(A'+t)$, then $a,a-t\in A'$ and hence every element of $A'_t$ is contained in one of these progressions $Q_j$ of size at least two. This provides the following lower bound for the total size of $\bigcup_jQ_j$: \begin{equation}\label{eq:Q_jsize}
\sum_{j=1}^r|Q_j|\geqslant |A'_t|.
\end{equation}
By assumption, every arithmetic progression $Q\subset A'$ has size $|Q|\leqslant L$, so $|Q_j|\leqslant L$ for all $j$. Combining this with \eqref{eq:B_tL} and \eqref{eq:Q_jsize} produces the inequality
$|B'_t|\geqslant r
    \geqslant \frac{1}{L}\sum_{j=1}^r |Q_j|\geqslant |A'_t|/L$.
\end{proof}
Ruzsa's Corollary \ref{cor:ruzsachowla} implies that, under the assumption that $\lVert \hat{1}_A\rVert_{\min}\leqslant K$, the largest progression contained in $A$ has size $O(K^2)$. By \eqref{eq:A_tsize}, we can find a $t\neq0$ so that $A_t=A\cap(A+t)$ has size $|A_t|\geqslant n/(2K)$. Hence, Lemma~\ref{prop:B_tlarge} shows that $B_t=A_t\setminus-A_t$ is also rather large:
\begin{equation}\label{eq:B_tlarge}
    |B_t|\gg \frac n{K^3}.
\end{equation} We then apply Proposition~\ref{prop:absoluteboundedextra} with this $t$ to deduce that \begin{align}\label{eq:1_Bbound}
    &|11(\hat{1}_{B_t}-\hat{1}_{-B_t})(x)-(\hat{1}_{C_t}-\hat{1}_{-C_t})(x)| \\
    &\leqslant 8\cdot\hat{1}_{D_t}(x)+ 2(\hat{1}_{B_t}+\hat{1}_{-B_t})(x)+O(K^{3}),\quad \forall x\in \To.\nonumber
\end{align}
The next proposition allows us to produce a lower bound for $K$ from these two facts. We have stated it in a rather general form, as we shall need it again to prove Theorem~\ref{th:generalcoefintro} in Section~\ref{sec:generalS}.
\begin{proposition}\label{prop:h*htrick}
    Let $P_1,P_2\in L^2(\mathbf{T})$, let $B\subset\mathbf{Z}$, and let $c,L>0$ be constants. Suppose that $|P_1(x)|\leqslant P_2(x)+L$ holds for almost all $x\in \To$, and that
    \begin{align}\label{eq:P_jcoef}
\begin{alignedat}{2}
    \Re \,\widehat{P_1}(b)   &\geqslant 1+c,   &\quad& \forall b\in B,\\
    |\widehat{P_2}(m)| &\leqslant 1,     &\quad& \forall m\in\mathbf{Z}.
\end{alignedat}
\end{align}
    Then $L\geqslant c|B|/\lVert\hat{1}_B\rVert_1^2$.
\end{proposition}
\begin{proof}
For notational convenience, we write $h(x):=|\hat{1}_{B}(x)|$ and note that $h\in L^\infty(\mathbf{T})$ as the assumptions of the proposition imply that $B$ is finite. Observe that by Parseval, 
\begin{align}\label{eq:hparseval}
    \sum_m |\hat{h}(m)|^2=\int_\To |h(x)|^2\,dx = \int_\To |\hat{1}_{B}(x)|^2\,dx= |B|.
\end{align} Since all Fourier coefficients of $\hat{1}_{B}$ are $0$ or $1$, we note that $(\hat{1}_{B}*\hat{1}_{B})(x)=\hat{1}_{B}(x)$. This implies the useful fact that $|\hat{1}_{B}(x)|\leqslant (h*h)(x)$ holds for all $x\in \To$, because 
\begin{equation}\label{eq:absoconv}
|\hat{1}_{B}(x)|=|\hat{1}_{B}*\hat{1}_{B}(x)|\leqslant (|\hat{1}_{B}|*|\hat{1}_{B}|)(x)=(h*h)(x).
\end{equation}
The first assumption in \eqref{eq:P_jcoef} and Parseval show that
\begin{align*}
    (1+c)|B|&\leqslant\Re\sum_{b\in B}\widehat{P_1}(b)\\
    &=\Re\int_\To P_1(x)\overline{\hat{1}_{B}(x)}\,dx.
\end{align*}
As $|P_1|\leqslant P_2+L$ a.e., we obtain the inequality
\begin{align*}
    (1+c)|B|\leqslant \int_\To \left(P_2(x)+L\right)|\hat{1}_{B}(x)|\,dx.\nonumber
\end{align*}
Note that both factors in the integrand are nonnegative functions; that the first term is nonnegative follows by the assumption that $|P_1|\leqslant P_2+L$. Using the observation \eqref{eq:absoconv} that $|\hat{1}_{B}(x)|\leqslant (h*h)(x)$ therefore allows us to get the upper bound
\begin{align}\label{eq:topolyboundextra}
    (1+c)|B|&\leqslant \int_\To P_2(x)(h*h)(x)\,dx+ L\int_\To (h*h)(x)\,dx.
\end{align}
For the contribution of the first term in \eqref{eq:topolyboundextra}, which we call $T_1$, we may use Parseval and the second assumption in \eqref{eq:P_jcoef} that $|\widehat{P_2}(m)|\leqslant 1$ for all $m$ to obtain the bound
\begin{align}
T_1&=\sum_{m\in \mathbf{Z}}\hat{h}(m)^2\widehat{P_2}(-m)\nonumber\\
&\leqslant \sum_{m\in \mathbf{Z}}|\hat{h}(m)|^2=|B|,\label{eq:T_1}
\end{align} where the final equality is \eqref{eq:hparseval}. To bound the contribution of the second term in \eqref{eq:topolyboundextra}, which we call $T_2$, we simply note that 
\begin{equation*}\label{eq:T_2}
    T_2= L\int_\To (h*h)(x)\,dx= L\left(\int_\To h(x)\,dx\right)^2=L\lVert \hat{1}_{B}\rVert_1^2.
\end{equation*}    
Finally, we may plug this estimate for $T_2$ and the estimate \eqref{eq:T_1} for $T_1$ into \eqref{eq:topolyboundextra}, showing that
$$(1+c)|B|\leqslant |B|+L\lVert\hat{1}_B\rVert_1^2,$$
and hence, $L\geqslant c|B|/\lVert \hat{1}_B\rVert_1^2$.
\end{proof}

By \eqref{eq:1_Bbound}, we may apply Proposition~\ref{prop:h*htrick} with $P_1= \frac{1}{8}\left(11(\hat{1}_{B_t}-\hat{1}_{-B_t})-(\hat{1}_{C_t}-\hat{1}_{-C_t})\right)$, $P_2=\hat{1}_{D_t}+ \frac14(\hat{1}_{B_t}+\hat{1}_{-B_t})$, $B=B_t$, $c=\frac38$, and $L=O(K^{3})$. This yields the bound 
\begin{equation}\label{eq:Kafterh*h}
    K^3\gg \frac{|B_t|}{\lVert \hat{1}_{B_t}\rVert_1^2}\gg \frac{n}{K^3\lVert \hat{1}_{B_t}\rVert_1^2},
\end{equation} by using the lower bound \eqref{eq:B_tlarge} for the size of $B_t$.
The final step is to find a good bound for the $L^1$-norm of $\hat{1}_{B_t}$.
\begin{lemma}\label{lem:betterl1}
Let $A'\subset\mathbf{Z}\setminus\{0\}$ be a finite symmetric set. Let $t\in \mathbf{Z}\backslash \{0\}$, and write $A'_t:=A'\cap(A'+t)$ and $B'_t:=A'_t\setminus-A'_t$. Then 
$$\| \hat{1}_{B'_t}\|_1 \ll \lVert \hat{1}_{A'}\rVert_1^3.$$
\end{lemma}
\begin{proof}
Let us begin by recalling that $-A'_t=A'_{-t}$ as $A'$ is symmetric, so $B'_t=(A'\cap(A'+t))\setminus(A'-t)$ and note that $B'_t$ is disjoint from $-B'_t=B'_{-t}$. By the triangle inequality, it suffices to prove the two bounds $\lVert \hat{1}_{B'_t}+\hat{1}_{-B'_t}\rVert_1,\lVert \hat{1}_{B'_t}-\hat{1}_{-B'_t}\rVert_1\ll \lVert \hat{1}_{A'}\rVert_1^3$. First, since $B'_t\cup -B'_t=\big(A'_t\cup A'_{-t}\big)\setminus A'\cap(A'+t)\cap (A'-t)$, we can write
\begin{equation}\label{eq:B_tcupB_-t}
\hat{1}_{B'_t}+\hat{1}_{-B'_{t}}= \hat{1}_{A'_t}+\hat{1}_{A'_{-t}}- 2\cdot\hat{1}_{A'\cap (A'+t)\cap (A'-t)}.
\end{equation}
As $A'_t=A'\cap(A'+t)$, we see that $\hat{1}_{A'_t}=\hat{1}_{A'}*\hat{1}_{A'+t}$, that $\hat{1}_{A'_{-t}}=\hat{1}_{A'}*\hat{1}_{A'-t}$, and that $\hat{1}_{A'\cap (A'+t)\cap (A'-t)}=\hat{1}_{A'}*\hat{1}_{A'+t}*\hat{1}_{A'-t}$. Note also that $\hat{1}_{A'+t}(x)=e(tx)\hat{1}_{A'}(x)$, so it is clear that $\lVert \hat{1}_{A'+t}\rVert_1=\lVert \hat{1}_{A'}\rVert_1$. Hence, Young's convolution inequality \eqref{eq:Young} gives the bound $\lVert \hat{1}_{A'_t}\rVert_1\leqslant\lVert\hat{1}_{A'}\rVert_1\times \lVert\hat{1}_{A'+t}\rVert_1=\lVert\hat{1}_{A'}\rVert_1^2$. Similarly, the other two functions $\hat{1}_{A'_{-t}}$ and $\hat{1}_{A'\cap (A'+t)\cap(A'-t)}$ have $L^1$-norm at most $\lVert\hat{1}_{A'}\rVert_1^3$, and hence so does \eqref{eq:B_tcupB_-t}.

\medskip

Next, to estimate $\lVert \hat{1}_{B'_t}-\hat{1}_{-B'_t}\rVert_1$, we observe that 
\begin{equation}\label{eq:B_t-B_-t}
   \hat{1}_{B'_t}-\hat{1}_{-B'_t}=\hat{1}_{A'}*\big(\hat{1}_{A'+t\setminus A'-t}-\hat{1}_{A'-t\setminus A'+t}\big). 
\end{equation}
Note that $\hat{1}_{A'+t\setminus A'-t}-\hat{1}_{A'-t\setminus A'+t}=(e(tx)-e(-tx))\hat{1}_{A'}$ and so has $L^1$-norm at most $2\lVert\hat{1}_{A'}\rVert_1$. Young's convolution inequality \eqref{eq:Young} then shows that
$$\lVert \hat{1}_{B'_t}-\hat{1}_{-B'_t}\rVert_1\leqslant \lVert \hat{1}_{A'}\rVert_1\times\lVert \hat{1}_{A'+t\setminus A'-t}-\hat{1}_{A'-t\setminus A'+t}\rVert_1\leqslant 2\lVert \hat{1}_{A'}\rVert_1^2.$$

\end{proof}
As $\lVert \hat{1}_A\rVert_1\ll K$ by Lemma~\ref{lem:mintoL^1}, applying Lemma~\ref{lem:betterl1} with $A'=A$ shows that $\lVert \hat{1}_{B_t}\rVert_1\ll K^3$. We may now plug this $L^1$-bound into \eqref{eq:Kafterh*h} to see
that $K^3\gg n/K^9$. So $K\gg n^{1/12}$, completing the proof of Theorem~\ref{th:chowlapoly} with the exponent $1/12$. We will improve this exponent to $1/5-o(1)$ in Section \ref{sec:exponentvalue}.

\section{One-sided estimates for cosine polynomials with general coefficients}\label{sec:generalS}

The purpose of this section is to prove Theorem~\ref{th:generalcoefintro}, establishing polynomial bounds for $K_S(n)$ whenever $S\subset\mathbf{R}\setminus\{0\}$ is a finite set of coefficients. The argument is quite similar to that of Theorem~\ref{th:chowlapoly}, though more involved. We restate the theorem here for the reader's convenience.
\begin{theorem}\label{th:generalcoef}
    Let $S\subset\mathbf{R}\setminus\{0\}$ be finite. Then there exist two constants $c_S,c'_S>0$ such that the following holds. For every symmetric set $A\subset\mathbf{Z}\setminus\{0\}$ of size $n=|A|$, and every choice of coefficients $s_a\in S$ satisfying $s_a=s_{-a}$ for all $a\in A$, we have that \begin{equation}\label{eq:generalcoef}
        \min_x \sum_{a\in A}s_ae(ax)\leqslant -c'_Sn^{c_S}.
    \end{equation}
\end{theorem}
\begin{proof}
    We shall prove this theorem by induction on the size of $S$. First, we note that Theorem~\ref{th:generalcoef} holds when $S=\{s\}$ is a singleton. Indeed, this is trivial when $s<0$ by evaluating at $x=0$, and when $s>0$ we may apply Theorem~\ref{th:chowlapoly}.

\medskip

By rescaling $S$ (and correspondingly scaling the final value of $c'_S$), we may assume that $\max_{s\in S}|s|=1$. We claim that, without loss of generality, we may additionally assume that $S=\{s_1>s_2>\dots>s_k\}\subset (0,\infty)$ consists of positive numbers which satisfy $s_1=1$ and $s_j<\frac{1}{1000}$ (say) for $2\leqslant j\leqslant k$. To see this, observe that if $A$ is a symmetric set $A\subset\mathbf{Z}\setminus\{0\}$ of size $n=|A|$ and $z_a\in S$ are some coefficients satisfying $z_a=z_{-a}$ for all $a\in A$ such that $$\sum_{a\in A}z_ae(ax)\geqslant -K,\quad\forall x\in\To,$$ then taking the $2\ell$-fold convolution and applying Lemma~\ref{lem:Kconv} shows that
\begin{equation}\label{eq:squareconv}
    \sum_{a\in A}z_a^{2\ell}e(ax)\geqslant -K^{2\ell},\quad \forall x\in \To.
\end{equation} Hence, we have shown that if we write $T:=S^{2\ell}=\{s^{2\ell}:s\in S\}$, then $S^{2\ell}$ is a set of at most $|S|$ positive numbers satisfying \begin{equation}\label{eq:Ksquareconv}
    K_{S}(n)\geqslant K_{S^{2\ell}}(n)^{1/(2\ell)}.
\end{equation} Since $\max_{s\in S}|s|=1$ and $S$ is finite, we may choose $\ell= O_S(1)$ such that $\max_{t\in T}t=1$ and all other elements of $T=S^{2\ell}$ are positive reals of size at most $\frac{1}{1000}$. We will show that $K_T(n)\gg_T n^{\Omega_T(1)}$ for sets $T$ satisfying these additional assumptions, and note that as $\ell=O_S(1)$, \eqref{eq:Ksquareconv} then implies the desired conclusion that $K_{S}(n)\gg_{S}n^{\Omega_{S}(1)}$ for arbitrary finite sets $S$.

\medskip

So let us now consider a fixed finite set $S=\{s_1=1>s_2>\dots>s_k\}\subset(0,1]$ of size $k\geqslant 2$ and for which $s_j<\frac{1}{1000}$ for all $2\leqslant j\leqslant k$. We may assume that the conclusion of Theorem~\ref{th:generalcoef} holds for all sets $S'$ of size $|S'|<k$. Let $A$ be an arbitrary fixed symmetric set $A\subset\mathbf{Z}\setminus\{0\}$ of size $n=|A|$, and let $z_a\in S$ be some coefficients satisfying $z_a=z_{-a}$ for all $a\in A$, such that $$F(x):=\sum_{a\in A}z_ae(ax)\geqslant -K, \quad\forall x\in \To.$$ Our goal is to show that $K\geqslant c'_Sn^{c_S}$. We begin by noting that we may partition $A=\bigsqcup_{j=1}^kA^{(j)}$ into disjoint symmetric sets such that
\begin{align}\label{eq:generalK}
    F(x)=\sum_{j=1}^ks_j\hat{1}_{A^{(j)}}(x)\geqslant -K.
\end{align}
\begin{claim}\label{cl:largeA_1}
    Without loss of generality, we may assume that $|A^{(1)}|\gg_Sn^{\Omega_S(1)}$.
\end{claim}
\begin{proof}[Proof of Claim \ref{cl:largeA_1}]
    By applying the induction hypothesis with the set $S':=S\setminus \{s_1\}$, we find that $$\min_x \sum_{j=2}^ks_j\hat{1}_{A^{(j)}}(x)\leqslant -c'_{S'}\left(n-|A^{(1)}|\right)^{c_{S'}}.$$ Hence, using a trivial upper bound for $s_1\hat{1}_{A^{(1)}}$ shows (recall that $s_1=1$): $$-K\leqslant\min_x F(x)\leqslant |A^{(1)}|-c'_{S'}\left(n-|A^{(1)}|\right)^{c_{S'}}.$$ So either $|A^{(1)}|\geqslant \frac{c'_{S'}}{100}n^{c_{S'}}$ which implies the claim, or else this immediately gives the desired lower bound $K\geqslant \frac{c'_{S'}}{2}n^{c_{S'}}\gg_Sn^{\Omega_S(1)}$.
\end{proof}

As in the previous section, we define two auxiliary functions 
\begin{align}\label{eq:f_tg_tdefi}
    F_t(x):=2(1+\sin(2\pi tx))F(x),\\
    G_t(x):=2(1-\sin(2\pi tx)) F(x),\nonumber
\end{align} where $t$ is some nonzero integer to be chosen later. The fact that $1\pm\sin(2\pi tx)$ is pointwise nonnegative implies the following lemma, analogous to Lemma~\ref{lem:basictextra}.
\begin{lemma}\label{lem:F_tG_tmin}
    For any $t\in\mathbf{Z}\setminus\{0\}$, the functions $F_t,G_t$ in \eqref{eq:f_tg_tdefi} satisfy $\lVert F_t\rVert_{\min},\lVert G_t\rVert_{\min}\leqslant 4K$.
\end{lemma}
As in \eqref{eq:Ffouriercoef}, the value of the Fourier coefficients $\widehat{F_t}(m),\widehat{G_t}(m)$ depends on which of the sets $A^{(1)},\dots, A^{(k)},\mathbf{Z}\setminus A$ each of $m,m-t,m+t$ lie in:
\begin{align}\label{eq:F_tG_tcoef}
    \widehat{F_t}(m)=2\sum_{j=1}^ks_j1_{A^{(j)}}(m)-i\sum_{j=1}^ks_j1_{A^{(j)}}(m-t)+i\sum_{j=1}^ks_j1_{A^{(j)}}(m+t)\\
    \widehat{G_t}(m)=2\sum_{j=1}^ks_j1_{A^{(j)}}(m)+i\sum_{j=1}^ks_j1_{A^{(j)}}(m-t)-i\sum_{j=1}^ks_j1_{A^{(j)}}(m+t).\nonumber
\end{align}
Let us write $\lambda_m:=\widehat{F_t}(m)$ for the Fourier coefficients of $F_t$, so note that $\widehat{G_t}(m)=\overline{\lambda_m}$ and that $\lambda_m=\overline{\lambda_{-m}}$ (as $F_t$ is real-valued) for all $m\in\mathbf{Z}$. We define the sets
\begin{align}\label{eq:A_tB_tC_tdefiextra}
    A^{(1)}_t&:=A^{(1)}\cap(A^{(1)}+t),\\
    B^{(1)}_t&:=A^{(1)}_t\setminus-A^{(1)}_t=(A^{(1)}\cap (A^{(1)}+t))\setminus(A^{(1)}-t),\nonumber\\
    C^{(1)}_t&:= (A^{(1)}+t)\setminus (A^{(1)}\cup (A^{(1)}-t)),\nonumber\\
    D^{(1)}_t&:= A^{(1)}\setminus(A^{(1)}_t\triangle-A^{(1)}_t).\nonumber
\end{align} 
which are the exact analogues of \eqref{eq:A_tB_tC_tdefi}, except that these are defined in terms of $A^{(1)}$, rather than the full Fourier spectrum $A$ of $F$.
Hence, unlike in \eqref{eq:f_tsimp}, the sets $\pm B^{(1)}_t,\pm C^{(1)}_t$ and $D^{(1)}_t$ do not necessarily cover the full Fourier spectrum of $F_t,G_t$, and we define $E_t:= (A\cup(A+t)\cup(A-t))\setminus ((\pm B^{(1)}_t)\cup(\pm C^{(1)}_t)\cup D_t^{(1)})$ to be the remaining part of the Fourier support of $F_t,G_t$. 
\begin{claim}\label{cl:lambda_m}
 There exist complex numbers $\varepsilon_m\in\mathbf{C}$, satisfying $|\varepsilon_m|\leqslant \frac{1}{100}$ for all $m\in\mathbf{Z}$, such that
 \begin{align}\label{eq:lambda_m}
     \lambda_m=\begin{cases}
         2\mp i+\varepsilon_m&\text{ if $m\in \pm B^{(1)}_t$,}\\
         2+\varepsilon_m&\text{ if $m\in D^{(1)}_t$,}\\
         \mp i+\varepsilon_m&\text{ if $m\in \pm C^{(1)}_t$,}\\
         \varepsilon_m &\text{ if $m\in  E_t$.}
     \end{cases}
 \end{align}
\end{claim}
\begin{proof}[Proof of Claim \ref{cl:lambda_m}]
This can be seen from \eqref{eq:F_tG_tcoef} as follows, upon recalling that $s_1=1$ while $|s_j|<1/1000$ for $2\leqslant j\leqslant k$. From the definitions \eqref{eq:A_tB_tC_tdefiextra}, it is clear that the three terms $2\cdot1_{A^{(1)}}(m),-i\cdot 1_{A^{(1)}}(m-t)$ and $i\cdot 1_{A^{(1)}}(m+t)$ in \eqref{eq:F_tG_tcoef} coming from $A^{(1)}$ contribute precisely the claimed `main term' to $\lambda_m$ in \eqref{eq:lambda_m} depending on whether $m$ lies in $\pm B^{(1)}_t,\pm C^{(1)}_t,D^{(1)}_t$ or $E_t$. Each $\varepsilon_m$ accounts for the contribution from the sets $A^{(j)},j\in[2,k]$ and so is a sum of at most three terms of the form $2s_{j},is_{j'},-is_{j''}$ with $j,j',j''\in[2,k]$, which each have size at most $2/1000$. Hence, each $\varepsilon_m$ certainly has size at most $1/100$.

\end{proof}
Define the real numbers $\rho_m,\sigma_m$ by $\rho_m-i\sigma_m:= \lambda_m^3$ for all $m\in\mathbf{Z}$. We now prove a generalised version of Proposition~\ref{prop:absoluteboundedextra}.
\begin{claim}\label{cl:absoluteboundedS}
    We have that \begin{align*}
        \left|\sum_{m\in\mathbf{Z}}\sigma_me(mx)\right|\leqslant \sum_{m\in\mathbf{Z}}\rho_me(mx)+O(K^3),\quad\forall x\in\To.
    \end{align*}
\end{claim}
\begin{proof}[Proof of Claim \ref{cl:absoluteboundedS}]
    By Lemmas \ref{lem:F_tG_tmin} and \ref{lem:Kconv}, we obtain
\begin{align}\label{eq:f_tg_tminS}
    \lVert F_t^{(*3)}\rVert_{\min}\ll K^3,\\
    \lVert G_t^{(*3)}\rVert_{\min}\ll K^3.\nonumber
\end{align} As we defined $\lambda_m=\widehat{F_t}(m)=\overline{\widehat{G_t}(m)}$, we see that these convolutions have the following explicit expressions: 
\begin{align*}
     F_t^{(*3)}(x)=&\sum_{m\in\mathbf{Z}}\lambda_m^3e(mx),\nonumber\\
     G_t^{(*3)}(x)=&\sum_{m\in\mathbf{Z}}\overline{\lambda_m}^3e(mx).\nonumber
\end{align*} Recall that $\lambda_m^3=\rho_m-i\sigma_m$ by definition, and hence $\overline{\lambda_m}^3=\rho_m+i\sigma_m$. Let us use this to rewrite our expressions for $F_t^{(*3)}$ and $G_t^{(*3)}$, so that upon recalling \eqref{eq:f_tg_tminS}, we have shown that
\begin{align*}
     &\sum_m\rho_me(mx)-i\sum_m\sigma_me(mx)\geqslant -O(K^3), \quad \forall x\in \To
\end{align*}
and
\begin{align*}
     &\sum_m\rho_me(mx)+i\sum_m\sigma_me(mx)\geqslant -O(K^3), \quad \forall x\in \To.
\end{align*}
Note that these two inequalities have precisely the same terms on their left-hand sides, except that the sign of the term $i\sum_m\sigma_me(mx)$ is flipped. Observe also that this term is a real-valued function on $\To$, because $\lambda_m=\overline{\lambda_{-m}}$ for all $m$ and hence $\sigma_{-m}:=-\Im(\lambda_{-m}^3)=\Im(\lambda^3_m)=-\sigma_m$. Hence, they combine to show the desired inequality that 
\begin{align*}
    \left|\sum_{m\in\mathbf{Z}}\sigma_me(mx)\right|\leqslant \sum_{m\in\mathbf{Z}}\rho_me(mx)+O(K^3),\quad\forall x\in\To.
    \end{align*}
\end{proof}
To apply Proposition~\ref{prop:h*htrick}, we need to check that some of the Fourier coefficients $\sigma_m$ are strictly larger than $\max_m|\rho_m|$. By cubing the expressions from Claim \ref{cl:lambda_m}, we observe that 
\begin{align*}\label{eq:lambda_m^3}
     \lambda_m^3=\begin{cases}
         2\mp 11i+\delta_m&\text{ if $m\in \pm B^{(1)}_t$,}\\
         8+\delta_m&\text{ if $m\in D^{(1)}_t$,}\\
         \pm i+\delta_m&\text{ if $m\in \pm C^{(1)}_t$,}\\
         \delta_m &\text{ if $m\in  E_t$,}
     \end{cases}
 \end{align*}
 for some complex numbers $\delta_m$ of size $|\delta_m|<1/2$ (say). Hence, from the fact that $\lambda_m^3=\rho_m-i\sigma_m$, it is clear that $\sigma_m\geqslant 10$ for all $m\in B_t^{(1)}$, while $|\rho_m|\leqslant 9$ for all $m\in\mathbf{Z}$. We also just showed in Claim \ref{cl:absoluteboundedS} that
 \begin{align*}
        \left|\sum_{m\in\mathbf{Z}}\sigma_me(mx)\right|\leqslant \sum_{m\in\mathbf{Z}}\rho_me(mx)+O(K^3),\quad\forall x\in\To.
    \end{align*}
These two ingredients allow us to apply Proposition~\ref{prop:h*htrick} with $P_1=\frac{1}{9}\sum_{m\in\mathbf{Z}}\sigma_me(mx)$, $P_2=\frac19\sum_{m\in\mathbf{Z}}\rho_me(mx)$, $B=B_t^{(1)}$, $c=1/9$, and $L=O(K^3)$. This yields the bound
\begin{equation}\label{eq:Kafterh*hS}
    K^3\gg\frac{|B_t^{(1)}|}{\left\lVert \hat{1}_{B_t^{(1)}}\right\rVert_1^2}.
\end{equation}
The final step in the proof of Theorem~\ref{th:generalcoef} is to show that we may pick a nonzero $t$ such that $|B^{(1)}_t|\geqslant n^{\Omega_S(1)}/K^{O_S(1)}$ and $\left\lVert \hat{1}_{B_t^{(1)}}\right\rVert_1\leqslant K^{O_S(1)}$. We first need to obtain an $L^1$-bound for $\hat{1}_{A^{(1)}}$.
\begin{claim}\label{cl:B^{(1)}L^1}
     We have that $\lVert \hat{1}_{A^{(1)}}\rVert_1\leqslant K^{O_S(1)}$.
\end{claim}
\begin{proof}[Proof of Claim \ref{cl:B^{(1)}L^1}]
Recall that we started our proof with the assumption \eqref{eq:generalK} that $\lVert F\rVert_{\min}\ll K$, where $F(x)=\sum_{j=1}^{|S|}s_j\hat{1}_{A^{(j)}}(x)$. By Lemma~\ref{lem:Kconv}, this implies that $\lVert F^{(*\ell)}\rVert_{\min}\ll_S K^{|S|}$ for all $1\leqslant \ell\leqslant |S|$. Lemma~\ref{lem:mintoL^1} then produces the $L^1$-bounds $\lVert F^{(*\ell)}\rVert_1\ll_S K^{|S|}$ for these convolutions of $F$. These convolutions have the following Fourier series:
$$F^{(*\ell)}(x)=\sum_{j=1}^{|S|}s_j^{\ell}\hat{1}_{A^{(j)}}(x).$$ The invertibility of the Vandermonde matrix $(s_j^\ell)_{j,\ell\in[|S|]}$ shows that we may find some real numbers $c_\ell,\ell\in[|S|]$, which clearly also have size $c_\ell=O_S(1)$, such that $$\sum_{\ell=1}^{|S|}c_\ell s_j^\ell=\begin{cases}
    1&\text{ if $j=1$}\\
    0&\text{ if $2\leqslant j\leqslant |S|$.}
\end{cases}$$
Hence, $\hat{1}_{A^{(1)}}= \sum_{\ell=1}^{|S|}c_\ell F^{(*\ell)}$ can be expressed as a linear combination of the convolutions $F^{(*\ell)},\ell\in [|S|]$, which each have $L^1$-norm at most $K^{O_S(1)}$. This confirms our claim that $\lVert \hat{1}_{A^{(1)}}\rVert_1\leqslant K^{O_S(1)}$.
    
\end{proof}
By plugging this $L^1$-bound from Claim \ref{cl:B^{(1)}L^1} into Lemma~\ref{lem:betterl1} (applied with $A'=A^{(1)}$), we deduce that $\lVert\hat{1}_{B_t^{(1)}}\rVert_1\ll \lVert \hat{1}_{A^{(1)}}\rVert_1^3\ll_SK^{O_S(1)}$ for every nonzero $t$. Thus, combining this with \eqref{eq:Kafterh*hS}, we have shown that 
\begin{equation}\label{eq:generalSfinalK}
K^{O_S(1)}\gg |B_t^{(1)}|,
\end{equation}
for any nonzero $t$.
So, to finish our proof of Theorem~\ref{th:generalcoef}, it only remains to find some nonzero $t$ for which $|B_t^{(1)}|\gg_Sn^{\Omega_S(1)}/K^{O_S(1)}$, as this combines with the previous inequality to show the desired bound $K\gg_Sn^{\Omega_S(1)}$. We achieve this final task in the next lemma. Recall that by Claim \ref{cl:largeA_1}, we may indeed assume that $|A^{(1)}|\gg_Sn^{\Omega_S(1)}$.
\begin{lemma}\label{lem:largeB_tS}
    Let $F=\sum_{j=1}^{|S|}s_j\hat{1}_{A^{(j)}}$ satisfy \eqref{eq:generalK}, and suppose that $|A^{(1)}|\gg_Sn^{\Omega_S(1)}$. Then there exists a nonzero $t\in\mathbf{Z}$ such that $|B^{(1)}_t|\gg_Sn^{\Omega_S(1)}/K^{O_S(1)}$.
\end{lemma}
\begin{proof}[Proof of Lemma~\ref{lem:largeB_tS}]
    H\"older's inequality shows that $\lVert \hat{1}_{A^{(1)}}\rVert_4^{2/3}\lVert\hat{1}_{A^{(1)}}\rVert_1^{1/3}\geqslant \lVert \hat{1}_{A^{(1)}}\rVert_2$. Now by Parseval, $\lVert\hat{1}_{A^{(1)}}\rVert_2=|{A^{(1)}}|^{1/2}$. Parseval also shows that $$\lVert\hat{1}_{A^{(1)}}\rVert_4^4=\int_\To(\hat{1}_{A^{(1)}})^2(\hat{1}_{-A^{(1)}})^2\,dx=\#\{x_1,x_2,x_3,x_4\in A^{(1)}:x_1-x_2=x_3-x_4\}.$$
    Using the $L^1$-bound $\lVert\hat{1}_{A^{(1)}}\rVert_1\ll K^{|S|}$ from Claim \ref{cl:B^{(1)}L^1} then yields
    $$\#\{x_1,x_2,x_3,x_4\in A^{(1)}:x_1-x_2=x_3-x_4\}\geqslant \frac{\lVert\hat{1}_{A^{(1)}}\rVert_2^6}{\lVert\hat{1}_{A^{(1)}}\rVert_1^2}\gg \frac{|{A^{(1)}}|^{3}}{K^{O_S(1)}}.$$
    The quantity on the left-hand side is the additive energy of $A^{(1)}$, and a standard computation reveals that 
    \begin{align*}
        &\#\{x_1,x_2,x_3,x_4\in A^{(1)}:x_1-x_2=x_3-x_4\}\\
        &=\sum_{t\in\mathbf{Z}}|A^{(1)}\cap(t+A^{(1)})|^2\\
        &\leqslant |A^{(1)}|^2+\left(\max_{t\neq 0}|A^{(1)}\cap(t+A^{(1)})|\right)\sum_m|A^{(1)}\cap(m+A^{(1)})|.
    \end{align*} As $\sum_m|A^{(1)}\cap(m+A^{(1)})|=|A^{(1)}|^2$, the previous two inequalities show that we may find some nonzero $t$ such that $A^{(1)}_t=A^{(1)}\cap(t+A^{(1)})$ has size
    \begin{equation}\label{eq:A1tlarge}
        |A^{(1)}_t|\gg \frac{|A^{(1)}|}{K^{O_S(1)}}.
    \end{equation}
    Lemma~\ref{lem:ruzsachowlaS} implies that the largest arithmetic progression contained inside $A^{(1)}$ has size $O_S(K^2)$. Hence, we may apply Lemma~\ref{prop:B_tlarge} with the set $A'=A^{(1)}$ to deduce that \begin{align*}
        |B_t^{(1)}|&\gg_S\frac{|A^{(1)}_t|}{K^2}\\
        &\gg |A^{(1)}|/K^{O_S(1)}\\
        &\gg_S n^{\Omega_S(1)}/K^{O_S(1)},
    \end{align*}
    where the second inequality is \eqref{eq:A1tlarge}, and the third follows from our assumption that $|A^{(1)}|\gg_S n^{\Omega_S(1)}$. This finishes the proof of the lemma.
\end{proof}
So by Lemma~\ref{lem:largeB_tS}, we can find a nonzero $t\in\mathbf{Z}$ such that $|B^{(1)}_t|\gg_Sn^{\Omega_S(1)}/K^{O_S(1)}$. Plugging this into \eqref{eq:generalSfinalK} shows that $K\gg_S n^{\Omega_S(1)}$, completing the proof of Theorem~\ref{th:generalcoef}.
\end{proof}
\section{An improved exponent}\label{sec:exponentvalue}
We put some further effort into improving the value of the exponent that our method produces for Chowla's cosine problem. Suppose throughout this section that $A=-A\subset\mathbf{Z}\setminus\{0\}$ is a finite symmetric set of integers of size $n=|A|$, and that $K>0$ is a constant such that
$$
\hat{1}_A(x)+K\geqslant 0,\quad \forall x\in\To.
$$
We shall reuse many of the ideas and notation from Section~\ref{sec:chowla1/12}, and the reader may wish to recall the definitions \eqref{eq:A_tB_tC_tdefi} of the sets $A_t,B_t,C_t,D_t$. We begin by recording a lemma which provides stronger information about $\hat{1}_{B_t}$ than the simple $L^1$-bound from Lemma~\ref{lem:betterl1}. The idea behind this lemma is partially inspired by the method of Bourgain \cite{Bourgainchowla1}.

\begin{lemma}\label{lem:Q_1Q_2}
    Let $A'\subset\mathbf{Z}\setminus\{0\}$ be a finite symmetric set such that $\hat{1}_{A'}+L\geqslant 0$. Let $t\in \mathbf{Z}\backslash \{0\}$, and write $A'_t:=A'\cap(A'+t)$ and $B'_t:=A'_t\setminus-A'_t$. Then we can write
$$ \hat{1}_{B'_t}(x)= Q_1(x)+Q_2(x),$$
where $Q_1,Q_2:\mathbf{T}\to\mathbf{C}$ satisfy $|Q_1(x)|\leqslant 4(\hat{1}_{A'}(x)+L)$ for all $x\in\mathbf{T}$, and $\lVert Q_2\rVert_2\ll L^2$.
\end{lemma}
\begin{rmk*}
This lemma is stronger than Lemma~\ref{lem:betterl1} in that it provides not only an $L^1$-bound for $\hat{1}_{B_t}$, but it also shows that almost all of the $L^2$-mass of $\hat{1}_{B_t}$ comes from a function $Q_1$ whose $L^1$-norm is even smaller, i.e. of size $O(K)$.
\end{rmk*}
\begin{proof}
As $\hat{1}_{B'_t}=\frac{1}{2}(\hat{1}_{B'_t}+\hat{1}_{-B'_t})+\frac12(\hat{1}_{B'_t}-\hat{1}_{-B'_t})$, it suffices to show that
\begin{align}\label{eq:Rdefi}
    \hat{1}_{B'_t}+\hat{1}_{-B'_t}&= R_1+R_2\\
    \hat{1}_{B'_t}-\hat{1}_{-B'_t}&=S_1+S_2,\label{eq:Sdefi}
\end{align} where $|R_1(x)|,|S_1(x)|\leqslant 4(\hat{1}_{A'}(x)+L)$, and $\lVert R_2\rVert_2,\lVert S_2\rVert_2\ll L^2$. The assumption that $\hat{1}_{A'}+L\geqslant 0$ implies that we may write $\hat{1}_{A'}(x)=T_1(x)+T_2(x)$ where $T_1(x)=\max(\hat{1}_{A'},0)$ satisfies $|T_1(x)|\leqslant \hat{1}_{A'}(x)+L$, and $T_2(x)=\min(\hat{1}_{A'},0)$. In particular, this implies that $\lVert T_1\rVert_1\leqslant \int_\To (\hat{1}_A+L)\,dx=L$ and that $\lVert T_2\rVert_\infty\leqslant L$. Exactly as in \eqref{eq:B_tcupB_-t}, we have that
\begin{equation}\label{eq:B_tcupB_-tnew}
\hat{1}_{B'_t}+\hat{1}_{-B'_{t}}= \widehat{1}_{A'_t}+\widehat{1}_{A'_{-t}}- 2\cdot\widehat{1}_{A'\cap (A'+t)\cap (A'-t)}.
\end{equation}
Note that $\hat{1}_{A'_t}=\hat{1}_{A'}*\hat{1}_{A'+t}$, that $\hat{1}_{A'_{-t}}=\hat{1}_{A'}*\hat{1}_{A'-t}$, and that $\hat{1}_{A'\cap (A'+t)\cap (A'-t)}=\hat{1}_{A'}*\hat{1}_{A'+t}*\hat{1}_{A'-t}$. Hence, as $\hat{1}_{A'+t}(x)=e(tx)\hat{1}_{A'}(x)$, we see that $\hat{1}_{A'_t}=(T_1+T_2)*(e(t\cdot)T_1+e(t\cdot)T_2)$. We can bound $$|T_1*(e(t\cdot)T_1)|\leqslant |T_1|*|T_1|\leqslant (\hat{1}_{A'}+L)*(\hat{1}_{A'}+L)=\hat{1}_{A'}+L^2,$$ and hence it is clear that we may split $T_1*(e(t\cdot)T_1)=U+V$ where $|U(x)|\leqslant \hat{1}_{A'}(x)+L$ and $\lVert V\rVert_\infty\ll L^2$. Young's convolution inequality \eqref{eq:Young} shows that the three remaining convolutions $T_1*[e(t\cdot)T_2], T_2*[e(t\cdot) T_1]$ and $ T_2*[e(t\cdot )T_2]$ have $L^\infty$-norm at most $L^2$; for example, $\lVert T_1*(e(t\cdot)T_2)\rVert_\infty\leqslant \lVert T_1\rVert_1\lVert T_2\rVert_\infty\leqslant L^2$. This shows that we may write $$\hat{1}_{A'_t}=R'_1+R'_2$$
where $R'_1:=U$ satisfies $|R'_1(x)|\leqslant \hat{1}_{A'}(x)+L$, and $R'_2$ consists of $V$ and the remaining three convolutions, so $\lVert R'_2\rVert_\infty \ll L^2$. A very similar argument provides an analogous expression $\hat{1}_{A'_{-t}}=R''_1+R''_2$ where $|R''_1(x)|\leqslant \hat{1}_{A'}(x)+L$ and $\lVert R''_2\rVert_\infty\ll L^2$. 

\medskip

To deal with the remaining term $\widehat{1}_{A'\cap (A'+t)\cap (A'-t)}$ in \eqref{eq:B_tcupB_-tnew}, we begin by using the above results to write
\begin{align*}
    \widehat{1}_{A'\cap (A'+t)\cap (A'-t)}&=\hat{1}_{A'}*\hat{1}_{A'+t}*\hat{1}_{A'-t}\\
    &=\hat{1}_{A'_t}*\big[e(-t\cdot)\hat{1}_{A'}\big]\\
    &=(R'_1+R'_2)*\big[e(-t\cdot)\hat{1}_{A'}]\\
    &=R'_1*\big[e(-t\cdot)T_1+e(-t\cdot)T_2\big]+R'_2*\hat{1}_{A'-t},
\end{align*}
where we  recall that $|R'_1(x)|\leqslant \hat{1}_{A'}(x)+L$ and $\lVert R'_2\rVert_\infty \ll L^2$, and that we write $T_1=\max(\hat{1}_{A'},0), T_2=\min(\hat{1}_{A'},0)$. Similarly to before, we may then bound $$|R'_1*(e(-t\cdot)T_1)|\leqslant \big[\hat{1}_{A'}+L\big]*|T_1|\leqslant (\hat{1}_{A'}+L)*(\hat{1}_{A'}+L)=\hat{1}_{A'}+L^2,$$ so that we may again split $R'_1*(e(-t\cdot)T_1)=u(x)+v(x)$ where $|u(x)|\leqslant \hat{1}_{A'}(x)+L$ and $\lVert v\rVert_\infty\ll L^2$. Young's inequality shows that $\lVert R'_1*(e(-t\cdot)T_2)\rVert_\infty \leqslant \lVert R'_1\rVert_1\lVert e(-t\cdot)T_2\rVert_\infty\ll L^2$. Finally, we obtain a strong $L^2$-bound for $R'_2*\hat{1}_{A'-t}$:
$$ \lVert R'_2*\hat{1}_{A'-t}\rVert_2^2=\sum_{m\in A'-t}|\widehat{R'_2}(m)|^2\leqslant \lVert R'_2\rVert_2^2\leqslant \lVert R'_2\rVert_\infty^2\ll L^4$$. Hence, we can write $\widehat{1}_{A'\cap (A'+t)\cap (A'-t)}=R'''_1+R'''_2$ where $R'''_1=u(x)$ satisfies $|R'''_1|\leqslant \hat{1}_{A'}+L$, and $R'''_2= v +R'_1*(e(-t\cdot)T_2)+ R'_2*\hat{1}_{A'-t}$ satisfies $\lVert R'''_2\rVert_2\ll L^2$. So we have found the desired functions in \eqref{eq:Rdefi} by taking $R_1:=R'_1+R''_1-2R'''_1$ and $R_2:=R'_2+R''_2-2R'''_2$.

\bigskip

The argument obtaining the expression \eqref{eq:Sdefi} for $ \hat{1}_{B'_t}-\hat{1}_{-B'_t}$ is again very similar, and this time based on the identity \eqref{eq:B_t-B_-t}: $\hat{1}_{B'_t}-\hat{1}_{-B'_t}=\hat{1}_{A'}*\big(\hat{1}_{A'+t\setminus A'-t}-\hat{1}_{A'-t\setminus A'+t}\big)$. 
Note that 
\begin{align*}
\hat{1}_{A'+t\setminus A'-t}-\hat{1}_{A'-t\setminus A'+t}&=(e(tx)-e(-tx))\hat{1}_{A'}\\
&=(e(tx)-e(-tx))T_1(x)+(e(tx)-e(-tx))T_2(x)
\end{align*}
and so this function can be written as $S'_1+S'_2$ where $|S'_1(x)|\leqslant 2|T_1(x)|\leqslant2(\hat{1}_{A'}+L)$, and $\lVert S'_2\rVert_\infty\leqslant 2L$. Convolving this expression $\hat{1}_{A'+t\setminus A'-t}-\hat{1}_{A'-t\setminus A'+t}=S'_1+S'_2$ with $\hat{1}_{A'}=T_1+T_2$ and arguing as above yields the desired expression \eqref{eq:Sdefi} for $\hat{1}_{B'_t}-\hat{1}_{-B'_t}$.
\end{proof}

We now replace the argument based on taking a triple convolution in
Proposition~\ref{prop:absoluteboundedextra} by a version that requires only two convolutions. The details of this are more involved, but morally follow the same strategy. Fix $t\neq0$, write $a_m=1_A(m)$, and recall the
function $f_t(x)=2(1+\sin(2\pi tx))\hat{1}_A(x)$.
Its Fourier coefficients are
\begin{equation}\label{eq:ztdef}
    \widehat{f_t}(m)
       =2a_m-ia_{m-t}+ia_{m+t}.
\end{equation}
In Lemma \ref{lem:basictextra} we showed
that $f_t\geqslant-4K$. As $\widehat{f_t}(0)=\int_\To f_t=0$,
an application of Lemma \ref{lem:Kconv} gives the inequality
\begin{equation}\label{eq:ftsquarepositive}
    f_t*f_t+16K^2
      =(f_t+4K)*(f_t+4K)
      \geqslant0.
\end{equation}
Unlike the argument in Section \ref{sec:chowla1/12}, we again have to multiply this inequality by a suitable nonnegative factor, this time of the form $1-\cos(2\pi t x + \pi/4)$. So define
$$
    r_t(x)
      =\bigl(1-\cos(2\pi tx+\pi/4)\bigr)(f_t*f_t)(x),
$$
and let
$$
    \psi_t(x)=\frac{r_t(x)+r_t(-x)}2,
    \qquad
    \phi_t(x)=\frac{r_t(x)-r_t(-x)}2
$$
be the even and odd parts of $r_t$, respectively. By \eqref{eq:ftsquarepositive}, we have $r_t(x)\geqslant-32K^2$, and trivially the same lower bound holds for $r_t(-x)$. In combination, as $\psi_t$ is even and $\phi_t$ is odd, these show that $\psi_t(x)\pm \phi_t(x)+32K^2\geqslant 0$ for all $x$. Hence, as $\phi_t$ is real-valued, we deduce the seemingly stronger pointwise inequality
\begin{equation}\label{eq:quadraticmajorant}
    |\phi_t(x)|
       \leqslant
       \psi_t(x)+O(K^2),
    \qquad \forall x\in\To.
\end{equation}
Let
$$
    \rho_m:=\widehat{r_t}(m).
$$
Since
$$
    1-\cos(2\pi tx+\pi/4)
      =
      1-\frac{1+i}{2\sqrt2}e(tx)
       -\frac{1-i}{2\sqrt2}e(-tx),
$$
we see from the definition of $r_t$ that
\begin{equation}\label{eq:rhomdef}
    \rho_m
      =
      \widehat{f_t}(m)^2
      -\frac{1+i}{2\sqrt2}\widehat{f_t}(m-t)^2
      -\frac{1-i}{2\sqrt2}\widehat{f_t}(m+t)^2.
\end{equation}
As $\psi,\phi_t$ are real-valued and even and odd, respectively, and $\psi_t+\phi_t=r_t$, their Fourier coefficients can be found by taking the real and imaginary parts of those of $r_t$:
\begin{equation}\label{eq:psiphicoefficients}
    \widehat{\psi_t}(m)=\Re\rho_m,
    \qquad
    \widehat{\phi_t}(m)=i\Im\rho_m.
\end{equation}
To make use of the inequality \eqref{eq:quadraticmajorant}, we need to find bounds for the Fourier coefficients of $\psi_t$ and $\phi_t$. This is accomplished in the next lemma, whose proof is an explicit but somewhat involved calculation.
\begin{lemma}\label{lem:quadraticcoefficients}
For every $m\in\mathbf{Z}$,
$$
    \widehat{\psi_t}(m)
       \leqslant4+\frac1{\sqrt2},
    \qquad
    |\widehat{\phi_t}(m)|
       \leqslant4+2\sqrt2.
$$
Moreover, if $m\in B_t$, then $-\Im\rho_m\geqslant4+\sqrt2$.
\end{lemma}

\begin{proof}
We begin by observing that by \eqref{eq:ztdef} and \eqref{eq:rhomdef}, \begin{align}
    \rho_m&:=\widehat{r_t}(m)\nonumber\\
    &=
      \widehat{f_t}(m)^2
      -\frac{1+i}{2\sqrt2}\widehat{f_t}(m-t)^2
      -\frac{1-i}{2\sqrt2}\widehat{f_t}(m+t)^2\nonumber\\
      &=(2a_m-ia_{m-t}+ia_{m+t})^2-\frac{1+i}{2\sqrt2}(2a_{m-t}-ia_{m-2t}+ia_{m})^2
      -\frac{1-i}{2\sqrt2}(2a_{m+t}-ia_{m}+ia_{m+2t})^2,\label{eq:rhomexpa}
\end{align} where we recall that $a_m=1_A(m)$. The value of $\rho_m$ therefore depends only on the five numbers
$$
    a_{m-2t},a_{m-t},a_m,a_{m+t},a_{m+2t}\in\{0,1\}.
$$
For $\alpha,\beta,\gamma\in\{0,1\}^3$, we can simplify
$$
    (2\beta-i\alpha+i\gamma)^2
      =
      4\beta-(\gamma-\alpha)^2
      +4i\beta(\gamma-\alpha),
$$
so every square occurring in \eqref{eq:rhomexpa} belongs to $\{0,-1,4,3+4i,3-4i\}$.
We first consider the case where $m\in B_t$. By definition, $B_t=(A+t)\cap A\setminus(A-t)$, so
$$
    (a_{m-t},a_m,a_{m+t})=(1,1,0),
$$
and hence
$$
    (2a_m-ia_{m-t}+ia_{m+t})^2=(2-i)^2=3-4i.
$$
The values of $a_{m-2t}$ and $a_{m+2t}$ are not determined by the assumption that $m\in B_t$. Writing $u=a_{m-2t}, v=a_{m+2t}$,
we can calculate the values of the other two squares in \eqref{eq:rhomexpa} for each of the four possible assignments of $(u,v)=(a_{m-2t},a_{m+2t})$:
\begin{align*}
    (2a_{m-t}-ia_{m-2t}+ia_m)^2
      &=
      \begin{cases}
          3+4i,&u=0,\\
          4,&u=1,
      \end{cases}\\
      (2a_{m+t}-ia_m+ia_{m+2t})^2
      &=
      \begin{cases}
          -1,&v=0,\\
          0,&v=1.
      \end{cases}
\end{align*}
Substituting these four possibilities into \eqref{eq:rhomexpa}
gives all possible values of $\rho_m$ when $m\in B_t$:
\begin{align*}
    \left.\rho_m\right|_{(u,v)=(0,0)}
      &=
      (3-4i)
      -\frac{1+i}{2\sqrt2}(3+4i)
      +\frac{1-i}{2\sqrt2}=
      3+\frac{\sqrt2}{2}-i(4+2\sqrt2),\\[1mm]
    \left.\rho_m\right|_{(u,v)=(0,1)}
      &=
      (3-4i)
      -\frac{1+i}{2\sqrt2}(3+4i)=
      3+\frac{\sqrt2}{4}
      -i\left(4+\frac{7\sqrt2}{4}\right),\\[1mm]
    \left.\rho_m\right|_{(u,v)=(1,0)}
      &=
      (3-4i)
      -\frac{1+i}{2\sqrt2}\,4
      +\frac{1-i}{2\sqrt2}=
      3-\frac{3\sqrt2}{4}
      -i\left(4+\frac{5\sqrt2}{4}\right),\\[1mm]
    \left.\rho_m\right|_{(u,v)=(1,1)}
      &=
      (3-4i)
      -\frac{1+i}{2\sqrt2}\,4=
      3-\sqrt2-i(4+\sqrt2).
\end{align*}
In particular, all four imaginary parts have the same sign, and we can simply observe that the following claimed bound indeed holds:
$$
    -\Im\rho_m\geqslant4+\sqrt2
    \qquad\text{for every }m\in B_t.
$$
It remains to prove the two uniform bounds on $\widehat{\psi_t}$ and $|\widehat{\phi_t}|$. This can again be done by a similar explicit (somewhat tedious) check, so we shall be brief. For each
fixed value of the triple
$$
    (a_{m-t},a_m,a_{m+t})\in\{0,1\}^3,
$$
one may perform the same calculation as above for each of the four choices of $(a_{m-2t},a_{m+2t})\in\{0,1\}^2$ to determine all possible values of $\rho_m=\widehat{r_t}(m)$. This lets us determine $\widehat{\psi_t}(m)=\Re \rho_m$ and $\widehat{\phi_t}(m)= i\Im \rho_m$.
Doing so produces the following table, where the maxima in the last two
columns are taken over these four choices $(a_{m-2t},a_{m+2t})\in\{0,1\}^2$:
\[
\begin{array}{c|c|c}
(a_{m-t},a_m,a_{m+t})
    &\max\Re\rho_m
    &\max|\Im\rho_m|\\ \hline
(0,0,0)&1/\sqrt2&\sqrt2/4\\
(0,0,1)&-1-3\sqrt2/4&5\sqrt2/4\\
(0,1,0)&4+1/\sqrt2&\sqrt2/4\\
(0,1,1)&3+1/\sqrt2&4+2\sqrt2\\
(1,0,0)&-1-3\sqrt2/4&5\sqrt2/4\\
(1,0,1)&-2\sqrt2&5\sqrt2/4\\
(1,1,0)&3+1/\sqrt2&4+2\sqrt2\\
(1,1,1)&4+1/\sqrt2&3\sqrt2/4.
\end{array}
\]
The largest entries in the final two columns give the desired bounds
$$
    \widehat{\psi_t}(m)
       \leqslant4+\frac1{\sqrt2},
    \qquad
    |\widehat{\phi_t}(m)|
       \leqslant4+2\sqrt2.
$$
\end{proof}

Now note that by this Lemma \ref{lem:quadraticcoefficients}, the inequality \eqref{eq:quadraticmajorant} has a function $\phi_t$ on the left-hand side whose Fourier coefficients on $B_t$ have size at least $4+\sqrt{2}$, which is strictly larger than any Fourier coefficient of the function $\psi_t$ on the right-hand side. This puts us in a very similar position as in Section \ref{sec:chowla1/12}, where a crucial inequality with similar features was obtained in Proposition \ref{prop:absoluteboundedextra}. In Section \ref{sec:chowla1/12}, our approach proceeded by applying Proposition \ref{prop:h*htrick}. This would work here too, but to prove the following quantitatively better result we use a slightly modified version of that argument, using the stronger input from Lemma \ref{lem:Q_1Q_2}.
\begin{proposition}\label{prop:B_tllK4}
Let $A\subset\mathbf{Z}\setminus\{0\}$ be a finite symmetric set such that $\hat{1}_{A}+K\geqslant 0$. For $t\in \mathbf{Z}\backslash \{0\}$, let $A_t:=A\cap(A+t)$ and $B_t:=A_t\setminus-A_t$. Then $|B_t|\ll K^4$ for every $t\in\mathbf{Z}\setminus\{0\}$.
\end{proposition}

\begin{proof}
Fix $t\neq0$ and write $B=B_t$. Apply
Lemma~\ref{lem:Q_1Q_2} with $A'=A$ and $L=K$, so that
$$
    \hat{1}_B=Q_1+Q_2,
$$
where
$$
    |Q_1|\leqslant4(\hat{1}_A+K),
    \qquad
    \lVert Q_2\rVert_2\ll K^2.
$$
As $\hat{1}_B(-x)=\overline{\hat{1}_B(x)}$, we may without loss of generality assume that $Q_1(-x)=\overline{Q_1(x)}$, by replacing $Q_j$ by $\frac12(Q_j(x)+\overline{Q_j(-x)})$ for $j=1,2$. Put
$$
    H(x)=|Q_1(x)|.
$$
Then $H$ is real-valued and even, so $\widehat H(m)$ is real for
every $m$. We also have
\begin{equation}\label{eq:Q1norms}
    \lVert Q_1\rVert_1\leqslant\int_\To 4(\hat{1}_A+K)\,dx=4K,
    \qquad
    \lVert Q_1\rVert_2=\lVert \hat{1}_B-Q_2\rVert_2
       \leqslant |B|^{1/2}+O(K^2).
\end{equation}
We test \eqref{eq:quadraticmajorant} against the pointwise
nonnegative function $H*H$. Let
$$
    X
      :=
      \int_\To
      (\psi_t(x)+O(K^2))(H*H)(x)\,dx.
$$
Using Parseval, the uniform bound $\widehat{\psi_t}(m)\leqslant 4+\frac{1}{\sqrt2}$ from Lemma~\ref{lem:quadraticcoefficients}, and
\eqref{eq:Q1norms}, we obtain the upper bound
\begin{align}
    X
    &\leqslant
      \left(4+\frac1{\sqrt2}\right)
      \sum_{m\in\mathbf{Z}}\widehat H(m)^2
      +O(K^2)\lVert H\rVert_1^2
      \nonumber\\
    &=
      \left(4+\frac1{\sqrt2}\right)
      \lVert Q_1\rVert_2^2+
    O(K^2)\lVert Q_1\rVert_1^2
      \nonumber\\
    &\leqslant
      \left(4+\frac1{\sqrt2}\right)|B|
      +O\bigl(K^2|B|^{1/2}+K^4\bigr).
      \label{eq:quadraticXupper}
\end{align}

On the other hand, \eqref{eq:quadraticmajorant} and the pointwise
inequality
$$
    H*H
      =
      |Q_1|*|Q_1|
      \geqslant
      |Q_1*Q_1|
$$
give the lower bound
$$
    X
      \geqslant
      \left|
      \int_\To
      \phi_t(x)(Q_1*Q_1)(x)\,dx
      \right|.
$$
Since $Q_1=\hat{1}_B-Q_2$
and
$$
    \hat{1}_B*\hat{1}_B=\hat{1}_B,
$$
we deduce that
$$
    Q_1*Q_1
      =
      \hat{1}_B
      -2\hat{1}_B*Q_2
      +Q_2*Q_2.
$$
By Parseval and the final assertion of
Lemma~\ref{lem:quadraticcoefficients},
\begin{align*}
    \left|
    \int_\To
    \phi_t(x)\hat{1}_B(x)\,dx
    \right|
    &=\left|\sum_{m\in B}\widehat{\phi_t}(-m)\right|\\
    &=\left|
    i\sum_{m\in B}(-\Im\rho_m)
    \right|
    \geqslant
    (4+\sqrt2)|B|.
\end{align*}
Using the uniform bound $|\widehat{\phi_t}(m)|\leqslant 4+2\sqrt2$ from
Lemma~\ref{lem:quadraticcoefficients}, we see that the other two terms satisfy
$$
    \left|
    \int_\To
    \phi_t(x)(\hat{1}_B*Q_2)(x)\,dx
    \right|\leqslant (4+2\sqrt{2})\sum_{m\in \mathbf{Z}}1_B(m)|\widehat{Q_2}(m)|
    \ll
    |B|^{1/2}\lVert Q_2\rVert_2
    \ll
    K^2|B|^{1/2}
$$ by Cauchy-Schwarz, and
$$
    \left|
    \int_\To
    \phi_t(x)(Q_2*Q_2)(x)\,dx
    \right|\leqslant (4+2\sqrt{2})\sum_{m\in \mathbf{Z}}|\widehat{Q_2}(m)|^2
    \ll
    \lVert Q_2\rVert_2^2
    \ll
    K^4.
$$
Thus, in total, we get the lower bound $X
      \geqslant
      (4+\sqrt2)|B|
      -O\bigl(K^2|B|^{1/2}+K^4\bigr)$.
Comparing this to the upper bound \eqref{eq:quadraticXupper}, 
we obtain
$$
    |B|
      \ll
      K^2|B|^{1/2}+K^4,
$$
which implies the desired
$|B|\ll K^4$.

\end{proof}

In order to use the previous proposition to obtain a good bound for $K$, it only remains to find a value of $t$ for which $B_t$ is large. Inequality \eqref{eq:B_tlarge} in Section \ref{sec:chowla1/12} shows that we may find a $t$ such that $|B_t|\gg n/K^3$, which combines with Proposition \ref{prop:B_tllK4} to show that $K\gg n^{1/7}$. To strengthen this to $K\gg n^{1/5-o(1)}$, we need the next lemma which allows us to find a $t$ for which the better bound $|B_t|\gg n^{1-o(1)}/K$ holds.

\begin{lemma}\label{lem:largeB_tnew}
    Let $A\subset \mathbf{Z}\setminus\{0\}$ be a finite symmetric set satisfying the inequality $\hat{1}_A(x)+K\geqslant 0$ for all $x\in \To$. Then there exists a $t\in\mathbf{Z}\setminus\{0\}$ such that $|B_t|\gg n/(K(\log n)^4)$.
\end{lemma}
\begin{proof}
    Suppose that $|B_t|\leqslant  L$ for all $t\neq 0$, so our goal is to show that $L\gg n/(K(\log n)^4) $. We begin by recalling \eqref{eq:B_tL} in Lemma \ref{prop:B_tlarge} which states that, if $A=\bigsqcup P_i^{(t)}$ is the partition of $A$ into the minimal possible number of arithmetic progressions $P_i^{(t)}$ with common difference $t$, and $Q_1^{(t)},\dots,Q_{r(t)}^{(t)}$ are those $P_i^{(t)}$ of size at least two, then $|B_t|\geqslant r(t)$. Hence, our assumption implies that
    \begin{equation}\label{eq:nolarger(t)}
        r(t)\leqslant L,
    \end{equation}
    for all $t\neq 0$. Clearly, if $a\in A\cap(A+t)$, then $a,a-t\in A$ and hence $a$ lies in one of the progressions $Q_i^{(t)}$ of size at least $2$. This shows that for every $t\neq 0$ we have that
    \begin{align}\label{eq:AcapA+tsubQ_i}
        A\cap(A+t)\subset \bigcup_{i=1}^{r(t)}Q_i^{(t)}.
    \end{align}
    \begin{claim}\label{cl:largeQ(t)}
        Let $t\neq 0$, and let $M\in\mathbf{N}$ be a parameter. Then $|A\cap(A+jt)|\geqslant |A\cap(A+t)|-ML$ for every $j\in \{1,2,\dots, M\}$.
    \end{claim}
    \begin{proof}[Proof of Claim \ref{cl:largeQ(t)}]
        Consider the $r(t)$ many progressions $Q_i^{(t)}$ with common difference $t$ that are contained in $A$. For each member $x$ of a progression $Q_i^{(t)}$ which is not one of the first $M$ terms of $Q_i^{(t)}$, it is clear that $x,x-jt\in Q_i^{(t)}\subset A$ whenever $j\in [M]$. Hence, for every $j\in[M]$, each $Q_i^{(t)}$ contains at least $|Q_i^{(t)}|-M$ elements of $A\cap(A+jt)$ (note that this trivially holds for those $Q_i^{(t)}$ of size $|Q_i^{(t)}|\leqslant M$). Thus, in total, we can bound
        \begin{align*}
            |A\cap(A+jt)|&\geqslant \sum_{i=1}^{r(t)}(|Q_i^{(t)}|-M)
            = \left|\bigcup_{i=1}^{r(t)}Q_i^{(t)}\right|-Mr(t)
            \geqslant |A\cap(A+t)|-ML,
        \end{align*} using \eqref{eq:nolarger(t)} and \eqref{eq:AcapA+tsubQ_i}.
    \end{proof}
    
    \medskip
    
    By Roth's Lemma \ref{lem:roth}, we see that $\sum_{t\in A}|A\cap(A+t)|=\#\{(a,t)\in A^2:a-t\in A\}\geqslant n^2/(2K)$, so we may find a $t_0\neq 0$ such that $k:=|A\cap(A+t_0)|\geqslant n/(2K)$.
    \begin{claim}\label{cl:235}
        Let $M\in\mathbf{N}$ be a parameter. For any tuple $(\alpha_p)_p$ of nonnegative integers, indexed by the primes $p\in[M]$, which satisfy $\sum_{p\leqslant M}\alpha_p\leqslant k/(2ML)$, we have that $|A\cap(A+(\prod_{p\leqslant M}p^{\alpha_p})t_0)|\geqslant 1$.
    \end{claim}
    \begin{proof}[Proof of Claim \ref{cl:235}]
    It suffices to show that whenever $\ell$ is an arbitrary nonnegative integer and $\sum_{p\leqslant M}\alpha_p\leqslant \ell$, then \begin{align}\label{eq:235indu}
        |A\cap(A+(\prod_{p\leqslant M}p^{\alpha_p})t_0)|\geqslant k-ML\ell .
    \end{align} We prove this by induction on $\ell$. The base case where $\ell=0$ holds as $|A\cap(A+t_0)|=k$ by definition. Now suppose that \eqref{eq:235indu} holds whenever $\sum_{p\leqslant M}\alpha_p\leqslant \ell$. Then an application of Claim \ref{cl:largeQ(t)} with $t=(\prod_{p\leqslant M}p^{\alpha_p})t_0$ and $j=2$ gives the desired lower bound $$|A\cap(A+(2^{\alpha_2+1}\prod_{2<p\leqslant M}p^{\alpha_p})t_0)|\geqslant |A\cap(A+(\prod_{p\leqslant M}p^{\alpha_p})t_0)|-ML\geqslant k-ML(\ell+1) ,$$ where the final inequality uses the induction hypothesis \eqref{eq:235indu}. A completely analogous application of Claim \ref{cl:largeQ(t)} with $j=p$ instead produces the same bound if we increase any other $\alpha_p$ by $1$.
        
    \end{proof}
Claim \ref{cl:235} shows that $A\cap\big(A+(\prod_{p\leqslant M}p^{\alpha_p})t_0\big)\neq\emptyset$ whenever $\sum_{p\leqslant M}\alpha_p\leqslant k/(2ML)$, so certainly all the integers $(\prod_{p\leqslant M}p^{\alpha_p})t_0$ with $\max_p\alpha_p\leqslant k/(2M^2L)$ must lie in $A-A$. Together with the trivial bound $|A-A|\leqslant n^2$, this implies that $$n^2\geqslant |A-A|\geqslant \left(\frac{k}{2M^2L}\right)^{\pi(M)}\geqslant \left(\frac{n}{4M^2LK}\right)^{\pi(M)},$$
where $\pi(M)$ denotes the number of primes up to $M$, by using that $k=|A\cap(A+t_0)|\geqslant n/(2K)$. Rearranging shows that $$L\geqslant \frac{n^{1-2/\pi(M)}}{4M^2K},$$ and we may now choose the parameter $M\approx (\log n)^2$ (this is only slightly suboptimal) to obtain the desired result that $L\gg n/(K(\log n)^4)$.

\end{proof}

Lemma \ref{lem:largeB_tnew} shows that we may find a $t\neq 0$ such that $$|B_t|\gg \frac{n}{K(\log n)^4}\geqslant \frac{n^{1-o(1)}}{K}.$$ Combining this with Proposition \ref{prop:B_tllK4} gives $K^5\gg n^{1-o(1)}$. Hence $K\gg n^{1/5-o(1)}$, completing the proof of Theorem~\ref{th:chowlapoly}.

\bibliographystyle{plain}
\bibliography{Chowlareferences}
\bigskip

\noindent
{\sc Mathematical Institute, Andrew Wiles Building, University of Oxford, Radcliffe
Observatory Quarter, Woodstock Road, Oxford, OX2 6GG, UK.}\newline
\href{mailto:bedert.benjamin@gmail.com}{\small bedert.benjamin@gmail.com}
\end{document}